\documentclass{amsart}

\usepackage{verbatim}
\usepackage{amssymb}
\usepackage{amsbsy}
\usepackage{amscd}
\usepackage{amsmath}
\usepackage{amsthm}
\usepackage{amsxtra}
\usepackage[mathscr]{eucal}

\renewcommand{\Pr}{Pr}
\newcommand{\afft}{\mathfrak{t}}
\newcommand{\cE}{{}^{\vee}E}
\newcommand{\sfD}{\mathsf{D}}

\newcommand{\affQ}[1]{\mathbf{Q}_{#1}}
\newcommand{\finn}{\bar{\mathfrak{n}}}
\newcommand{\vecW}[1]{\mathsf{W}^{(#1)}}
\newcommand{\isomap}{\overset{\sim}{\rightarrow} }
\newcommand{\finW}{\bar{W}}

\newcommand{\Dyn}{\mathrm{Dyn}}
\newcommand{\W}{\mathscr{W}}
\newcommand{\BRST}{H^{\mathrm{BRST}}}

\newcommand{\nc}{\newcommand}
\nc{\Hp}[1]{H^{#1}}
\newcommand{\V}{\mathcal{V}}

\newcommand{\affQp}{Q_{(0)}}
\renewcommand{\r}{\mathfrak{r}}

\newcommand{\bCg}{\mathcal{C}^{\bullet}}%{\Vkg\* \Lamsemi{\bullet}(\L{\sg_{>0}})}

\newcommand{\Dg}{D}
\newcommand{\tri}{\triangle}
\newcommand{\affP}{P_{0,k}^+}
\newcommand{\affBGG}{\BGG_{0,k}}
\newcommand{\Vkg}{V^k(\fing)}

\newcommand{\prroots}{\Delta_+^{\rm re}}

\newcommand{\affn}{\mathfrak{n}}
\newcommand{\sW}{\bar{W}}

\newcommand{\rroots}{\Delta^{\rm re}}
\newcommand{\iroots}{\Delta^{\rm im}}
\newcommand{\h}{\mathfrak{h}}

\newcommand{\affh}{\mathfrak{h}}
\newcommand{\snn}{\sn_-}
\newcommand{\snp}{\sn_+}
\newcommand{\affg}{\mathfrak{g}}
\newcommand{\fing}{\bar{\mathfrak{g}}}
\newcommand{\sQ}{\bar{\mathbf{Q}}_-}

\newcommand{\sCl}{\bar{\mathscr{C}l}}
\newcommand{\Cl}{\mathscr{C}l}
\newcommand{\bCl}{\mathscr{C}l}
\newcommand{\affCl}{\mathscr{C}l}
\newcommand{\sL}{\bar{L}}
\newcommand{\slam}{\bar{\lam}}
\newcommand{\sP}{\bar{P}}
\newcommand{\sM}{\bar{M}}
\newcommand{\fdW}{\mathfrak{Fin}(\Wfin)}
\newcommand{\sH}{H^{\mathrm{Lie}}}
\newcommand{\Wfin}{\W^{\mathrm{fin}}(\sg,f)}

\newcommand{\sh}{\bar{\mathfrak{h}}}
\newcommand{\Wg}{\W^k(\sg, f)}
\newcommand{\schi}{\bar{\chi}}
\newcommand{\sI}{\bar{I}}
\newcommand{\sroots}{\bar{\roots}}
\newcommand{\wJ}{\widehat{J}}
\newcommand{\bh}{\mathfrak{h}}
\newcommand{\sn}{\bar{\mathfrak{n}}}
\newcommand{\sE}{\bar E}
\newcommand{\tp}{{\mathrm{top}}}
\newcommand{\st}{\mathrm{st}}
\newcommand{\Lamsemi}[1]{\bigwedge\nolimits^{\frac{\infty}{2}+#1}}
\newcommand{\semiLamp}[1]{\bigwedge\nolimits^{\frac{\infty}{2}+#1}(L\sg_{>0})}

\newcommand{\sBGG}{\BGG_{0}(\sg)}
\renewcommand{\L}[1]{L #1}

\newcommand{\gVerma}[1]{M_0(#1)}

\newcommand{\Zh}{\mathrm{Z}\mathrm{h}}

\newcommand{\sg}{\bar{\mathfrak{g}}}
\newcommand{\BGG}{{\mathcal O}}
\newcommand{\aff}{\mathrm{aff}}

\newcommand{\N}{\mathbb{N}}

\newcommand{\1}{{\mathbf{1}}}
\newcommand{\teigi}{=} %{\underset{\mathrm{def}}{=}}

\newcommand{\dual}[1]{{#1}^*}
\newcommand{\bra}{{\langle}}
\newcommand{\ket}{{\rangle}}

\newcommand{\roots}{\Delta}
\newcommand{\nno}{\nonumber}
\newcommand{\Lam}{\Lambda}

\newcommand{\affPp}{P_{0,+}^k}
\newcommand{\lam}{\lambda}
\newcommand{\ra}{\rightarrow}
\newcommand{\+}{\mathop{\oplus}}
\newcommand{\Z}{\mathbb{Z}}
\newcommand{\Mod}{\text{-}\mathfrak{mod}}
\newcommand{\adMod}{\text{-}\mathfrak{Mod}}

\newcommand{\inv}{^{-1}}
\newcommand{\bg}{\mathfrak{g}}
\renewcommand{\*}{{\otimes}}
\newcommand{\C}{\mathbb{C}}
\newcommand{\che}{^{\vee}}
\newcommand{\fin}{\mathrm{fin}}

\newtheorem{Th}{Theorem}[subsection]
\newtheorem{Pro}[Th]{Proposition}
\newtheorem{Lem}[Th]{Lemma}
\newtheorem{Def}[Th]{Definition}
\newtheorem{Co}[Th]{Corollary}

\newtheorem{Rem}[Th]{Remark}
\numberwithin{defn}{section}
\DeclareMathOperator{\Aut}{Aut}
\DeclareMathOperator{\im}{im}
\DeclareMathOperator{\rank}{rank}
\DeclareMathOperator{\ch}{ch}
\DeclareMathOperator{\Res}{Res}

\DeclareMathOperator{\id}{id}
\DeclareMathOperator{\Lie}{Lie}

\DeclareMathOperator{\End}{End}
\DeclareMathOperator{\gr}{gr}
\DeclareMathOperator{\Hom}{Hom}
\DeclareMathOperator{\Dim}{Dim}

\DeclareMathOperator{\ad}{ad}

\title[Representation Theory  of  $W$-Algebras, II]
{Representation Theory of $W$-Algebras, II:\\
Ramond Twisted Representations}

\author[Tomoyuki Arakawa]{Tomoyuki Arakawa}

\thanks{The author is partially supported 
by the JSPS Grant-in-Aid for Young Scientists (B)
No.\ 17740006}
\dedicatory{Dedicated to Professor Akihiro Tsuchiya on the occasion of his 
retirement from Nagoya University
}

\address{Department of Mathematics, 
Nara Women's University, Nara 630-8506, JAPAN}

\email{arakawa@cc.nara-wu.ac.jp}

\subjclass[2000]{}

\begin{document}
\begin{abstract}
We study 
the Ramond twisted representations of 
the affine $W$-algebra $\W^k(\sg,f)$
in the case that $f$ admits a good even grading.
We establish the 
vanishing 
and the almost irreducibility
of the corresponding 
BRST cohomology.
This confirms some of the recent conjectures of
Kac and Wakimoto \cite{KacWak}.
In type $A$,
our  results
give  the characters of
all  
 irreducible ordinary  
Ramond twisted representations of $\W^k(\mathfrak{sl}_n,f)$
for all nilpotent elements $f$
and
 all  non-critical  $k$,
%$(thanks to the result \cite{BruKle05} of Brundan and Kleshchev),
and prove the existence of modular invariant representations
 conjectured in  \cite{KacWak}.

\end{abstract}
\maketitle

\section{Introduction}
Let $\fing$ be a 
complex simple Lie algebra,
$f$ a nilpotent element of $\sg$,
$\affg$ the non-twisted affine Kac-Moody 
 Lie algebra
associated with $\sg$.
Let $\Wg$ be the 
{\em affine $W$-algebra}
associated with $(\sg, f)$
at level $k\in \C$,
  defined by 
the 
 method of the quantum
BRST
reduction \cite{FF90,BoeTji94,KacRoaWak03}.

%Unlike the $W$-algebra 
%$\W^k(\fing,f_{\pri})$ associated with
%the principal nilpotent 
The vertex algebra $\Wg$ is
in general  $\frac{1}{2}\Z_{\geq 0}$-graded \cite{KacWak04}.
Therefore 
it is natural \cite{KacWak}  to consider
its {\em{Ramond twisted representations}}\footnote{
If $f$ is an even nilpotent element 
then $\Wg$ is $\Z_{\geq 0}$-graded and Ramond twisted
representations are usual (untwisted) representations.}. 
In fact 
it is in the Ramond twisted representations
where
the corresponding {\em finite $W$-algebra}
$\Wfin$ \cite{Lyn79,BoeTji93,Pre02}
appears as its {\em Zhu algebra},
according to \cite{De-Kac06}.

In the previous paper \cite{Ara07} we studied the
representations  of
%the $W$-algebra
$\Wg$
in the case that $f$ is a principal nilpotent element.
In the present paper we study the Ramond twisted representations
of % the affine $W$-algebra
$\Wg$ in the case that $f$ admits a good even grading.
All nilpotent elements in type $A$ %of $\mathfrak{sl}_n$ 
satisfy this condition.

There is a 
natural
BRST (co)homology
functor $\BRST_{0}(?)$
from a suitable category of representations of 
$\affg$ at level $k$
to the category of 
Ramond twisted representations of $\Wg$.
In our case  $\BRST_{\bullet}(M)$ is essentially the same 
BRST cohomology
studied  
in the recent work \cite{KacWak} of Kac and Wakimoto.
In the case that $f$ is principal this functor is
identical to the ``$-$''-reduction functor studied in 
\cite{FKW92,Ara04,Ara07}.

The main
result
 of 
this paper is the 
vanishing 
and the {\em almost irreduciblity}
%(see \S \ref{subsection:Zhu-algbera})
of the BRST cohomology (Theorem \ref{Th:Main}).
Though our formulation is slightly different 
from that of  \cite{KacWak},
this result proves
Conjecture B of \cite{KacWak},
partially.
Here,
recall \cite{De-Kac06} that a positive energy   representation
$M=\bigoplus_{d\in d_0+\Z_{\geq 0}}M_0$,
$M_{d_0}\ne 0$, of a vertex algebra $V$ is  called
{almost irreducible}
if $M$ is generated by $M_{d_0}$ and there is no graded submodule
of $M$ intersecting $M_{d_0}$ trivially.
In particular 
an almost irreducible module $M$ is irreducible if and only if
its ``top part'' $M_{d_0}$ is irreducible over the Zhu algebra of $V$.

In our case 
the top part of the 
BRST cohomology functor is 
identical to the 
Lie algebra homology functor
(the Whittaker functor \cite{Mat90,BruKle05})
from the highest weight category of $\fing$
to the category of $\Wfin$-modules
(see \S \ref{subsection:top-part-of-BRST}).
Therefore our result reduces the study of
the BRST cohomology functor to that 
of the Whittaker functor in the representations theory
of finite $W$-algebras.

Although 
the representation theory of
finite $W$-algebras  %$\Wfin$
has been rapidly developing
(cf.\ \cite{Pre07,Pre06,Los07,BruGooKle08}),
not much is known 
 about
the Whittaker functor
%in
%the representation theory of finite $W$-algebras 
associated with $\Wfin$ %for general $\fing$
except for some special cases \cite{Mat90},
unless
 $\fing=\mathfrak{sl}_n$:
In type $A$,
Brundan and Kleshchev
\cite{BruKle05}
determined the characters of all  irreducible
finite-dimensional representations of $\W^{\fin}(\mathfrak{sl}_n,f)$,
by
showing that
the Whittaker functor sends an simple module to zero or a simple module,
and any simple $\W^{\fin}(\mathfrak{sl}_n,f)$-module
is obtained in this manner.
It follows that
in type $A$
the almost irreduciblity
of
the BRST cohomology
actually implies the irreduciblity,
and furthermore,
any
irreducible  ordinary\footnote{An
irreducible positive  energy 
representation of a vertex algebra is called ordinary
if its all homogeneous subspaces
 are finite-dimensional. } %Ramond twisted 
representation
of $\W^k(\mathfrak{sl}_n,f)$
 is 
isomorphic to $\BRST_0(L(\lam))$
for some irreducible highest weight representation
 $L(\lam)$ of $\widehat{\mathfrak{sl}}_n$ with highest weight $\lam$
 (Theorem \ref{Th:Main:typeA}).
Hence our result shows that
 the character of {\em every}
irreducible ordinary 
Ramond twisted representation 
of $\W^k(\mathfrak{sl}_n,f)$
at any level $k\in \C$ 
is  determined 
by that of the corresponding irreducible highest weight
representation of $\affg$,
which is known \cite{KasTan00} (in terms of the Kazhdan-Lusztig polynomials)
provided that $k$ is not critical.
This generalizes
 the main results of \cite{Ara05,Ara07}.

\smallskip 
%Among representations of vertex algebras,
The most important representations of a
vertex algebra  are those
irreducible 
ordinary
representations
whose normalized characters 
are modular invariant.
%invariant under the action of 
%Kac and Wakimoto \cite{KacWak}
%called the pair $(k,f)$
%{\em exceptional}
%if the corresponding $W$-algebra 
%$\Wg$ enjoy such representations
%and they studied the problem
%to classify the
%pairs $(k,f)$
%for which 
Kac and Wakimoto \cite{KacWak}
have
recently  discovered
the remarkable triples $(\fing,f,k)$,
for which 
the (nonzero) normalized
Euler-Poincar\'{e} characters of the BRST cohomology
$\BRST_{\bullet}(L(\lam))$,
with 
the coefficient in the
irreducible 
principal admissible   representations
$L(\lam)$
%\cite{KacWak89}
of $\affg$ at level $k$,
%tof $\widehat{\mathfrak{sl}}_n$ 
%enjoy this property.
%form a vector valued modular function.
are homomorphic functions
on the complex upper half plane
and span an $SL_2(\Z)$-invariant space\footnote{
In the case that $f$ is a principal nilpotent
the existence of
modular invariant representation of $\Wg$
was conjectured by
Frenkel, Kac and Wakimoto \cite{FKW92}
and proved in \cite{Ara07}.}.
%Though our formulation is slightly different 
%from that of  \cite{KacWak},
Our 
results show
in type $A$
that these 
Euler-Poincar\'{e} characters 
are indeed  characters of irreducible Ramond twisted representations
of $\W^k(\mathfrak{sl}_n,f)$,
as conjectured in \cite{KacWak}
%and thus
%confirm the existence of modular invariant
%representations of $\W^k(\mathfrak{sl}_n,f)$
(see Theorem \ref{Th:moruldar-invariant-represenrations})\footnote{It seems 
that
the ``top parts'' of  
 modular invariant representations
are in general  ``generic'' representations
of $\Wfin$,
 see Theorem \ref{Th;irr-generic}.
}.

\smallskip
Non-twisted representations of affine $W$-algebras
are studied in our
 subsequent paper.

\subsection*{Acknowledgment}
The author is grateful to 
Professor M. Wakimoto for explaining his joint work \cite{KacWak}
with
Professor V. Kac,
and to
Professor A.  Premet for 
very useful discussion on finite $W$-algebras.
He is grateful to Professor Simon Goodwin who pointed out a mistake
in the first version of the paper.
The results of the paper were reported in part
at TMS \& AMS Joint International Conference
in Taichung in December 2005,
in ``Representation Theory of Algebraic Groups and Quantum Groups 06"
in Nagoya in June  2006
and in ``Exploration of New Structures and Natural Constructions
in Mathematical Physics'' in Nagoya in March 2007.
%The author apologizes for the delay in putting 
%the results  on paper.

\bigskip

\noindent {\em Notation.}
Throughout this paper the ground field is the complex number $\C$
and tensor products and dimensions are always meant to be as vector spaces 
over $\C$.

\section{Preliminaries on Vertex Algebras and 
their Twisted Representations}
In this section 
we collect the  necessary information
on vertex algebras and their
(twisted) representations.
The textbook \cite{Kac98,FreBen04}
and the papers \cite{Li96-2,BakKac04,De-Kac06}
are our basic references in this section.
\subsection{Fields}\label{subsection:fields}
Let $V$ be a vector space.
For a formal series
$a(z)\in (\End V)[[z,z\inv]]$,
we set $a_{(n)}=\Res_z z^n a(z)$,
where $\Res_z$ denotes the coefficient of $z\inv$.

An element $a(z)=\sum_{n\in\Z}a_{(n)}z^{-n-1}\in (\End V)[[z,z\inv]]$
is called
a  {\em field} on $V$ if $a_{(n)}v=0$ for all $v\in V$ and $n\gg 0$.

The normally ordered product
\begin{align}
 :a(z)b(z):=a(z)_-b(z)+b(z)a(z)_+
\end{align}
of two fields $a(z)$ and $b(z)$ is also a field,
where $a(z)_-=\sum_{n<0}a_{(n)}z^{-n-1}$ and
 $a(z)_+=\sum_{n\geq 0}a_{(n)}z^{-n-1}$.

Two fields $a(z)$ and $b(z)$  are called {\em mutually 
 local} if 
\begin{align}
 (z-w)^r [a(z), b(w)]=0\quad \text{for }r\gg 0
\label{eq:locality}
\end{align}
in $(\End V)[[z,z\inv,w,w\inv]]$.

Set 
\begin{align}
 \delta(z-w)=\sum_{n\in \Z}z^n w^{-n-1} \in \C[[z,z\inv,w,w\inv]].
\end{align}
The locality 
 (\ref{eq:locality})
gives
\begin{align}
 [a(z),b(w)]=\sum_{n\geq 0}(a(z)_{(n)}b(w))\partial_w^{[n]}
\delta(z-w),
\end{align}
where
 $\partial_{w}^{[n]}=\partial_w^n /n!$,
$\partial_w=\frac{\partial}{\partial w}$,
and
\begin{align*}
 a(z)_{(n)}b(w)=\Res_{z}(z-w)^n [a(z),b(w)].
\end{align*}

\subsection{Vertex Algebras}
A {\em vertex algebra} is a vector space
$V$ equipped with the 
following data:
\begin{itemize}
 \item A vector $\1\in V$ (vacuum vector),
\item $T\in \End V$ (translation operator),
\item A collection $\{a^{\alpha}(z)
=\sum_{n\in \Z}
a^{\alpha}_{(n)}z^{-n-1}; \alpha\in A\}$ of fields on $V$,
where $A$ is an index set
(generating fields),
\end{itemize}
These data are subject  to the following:
\begin{enumerate}
 \item $T \1=0$,
\item $[T, a^{\alpha}(z)]=\partial_z a^{\alpha}(z)$ for all $\alpha
\in A$,
\item $a^{\alpha}(z)\1\in V[[z]]$
for all $\alpha
\in A$,
\item 
the vectors
$a^{\alpha_1}_{(m_1)}\dots a^{\alpha_r}_{(m_r)}\1$
with $r\geq 0$,
$\alpha_i\in A$
and $m_i\in \Z$
span $V$,
\item for any $\alpha, \beta\in A$
the fields $a^{\alpha}(z)$
and $a^{\beta}(z)$ mutually local.
\end{enumerate}

\smallskip

Let $V$ be a vertex algebra.
There exists a unique 
linear map 
\begin{align}
\label{eq:state-field}
 V\rightarrow (\End V)[[z,z\inv]],
\quad a \mapsto Y(a,z)=\sum_{n\in \Z}a_{(n)}z^{-n-1}
\end{align}
such that
\begin{enumerate}
 \item $Y(a,z)$ is a field on $V$ for any $a\in V$,
\item $Y(a,z)$ and $Y(b,z)$ are mutually
local for any $a,b\in V$,
\item $[T, Y(a,z)]=\partial_z a(z)$
for any $a\in V$,
\item $Y(a,z)\1\in V[[z]]$
and $\lim\limits_{z\ra 0}Y(a,z)\1=a$
 for any $a\in V$,
\item $Y(a^{\alpha}_{(-1)}\1,z)=a^{\alpha}(z)$
for any generating filed $a^{\alpha}(z)$.
\end{enumerate}
The map $Y(?,z)$
is called the  {\em state-field correspondence}.

\smallskip

A {\em Hamiltonian}
of a vertex algebra $V$
is a diagonalizable operator  
$H\in \End V$
such that
\begin{align*}
 [H,Y(a,z)]=Y(Ha,z)+z\partial_z Y(a,z)
\quad \text{for all }a\in V.
\end{align*}
A vertex algebra with a Hamiltonian
$H$ is called  graded.
If $a$ is a eigenvector of $H$
its eigenvalue  is called the {\em conformal weight}
of $a$ and denoted by $\Delta_a$.
Let\footnote{This differs from the notation in \cite{Ara07}.}\begin{align*}
V_{\Delta}=\{a\in V;
H a=\Delta a\},
   \end{align*}
so that
$V=\bigoplus_{\Delta\in \C}V_{\Delta}$.
\subsection{Twisted Representations of Vertex  Algebras}
\label{subsection:twisted-representations}
Let $N\in \N$.
An {\em $N$-twisted field}
 $a(z)$ on a  vector space $M$
is a formal power series  in $z^{1/N}$,
$z^{-1/N}$
of the form 
\begin{align}
 a(z)=\sum_{n\in \frac{1}{N}\Z}a_{(n)}z^{-n-1},
\quad a_{(n)}\in \End (M)
\end{align}
such that $a_{(n)}m=0$
for all $m\in M$ and $n\gg 0$.

{Two $N$-twisted fields $a(z)$ and $b(z)$ on $M$
are called mutually local if they satisfy (\ref{eq:locality})
in $(\End M)[[z^{1/N},z^{-1/N},w^{1/N},w^{-1/N}]]$.

Let $V$ be a vertex algebra,
$\sigma$ an automorphism of $V$
of order $N$.
A
{\em  $\sigma$-twisted representation}
 of
$V$
is a vector space $M$
equipped with a linear map
from $V$ to the space of $N$-twisted fields on $M$,
\begin{align*}
V \ra (\End M)[[z^{\frac{1}{N}},z^{-\frac{1}{N}}]],
\quad a\mapsto Y^M(a,z)=\sum_{n
\in \frac{1}{N}\Z}a_{(n)}^Mz^{-n-1},
\end{align*}
such that
\begin{align}
&Y^M(\sigma a,z)=Y^M(a, e^{2 \pi i}z),\\
&Y^M(\1,z)=\id_M,
\end{align}
and
\begin{align}
 \sum_{i=0}^{\infty}
&\begin{pmatrix}
 m\\i
\end{pmatrix}
(a_{(r+i)}b)^M_{(m+n-i)}
\label{eq:Borcherds-id}\\
&=\sum_{i=0}^{\infty}
(-1)^i\begin{pmatrix}
 r\\i
      \end{pmatrix}
\left(a^{M}_{(m+r-i)}b^M_{(n+i)}
-(-1)^r b^M_{(n+r-i)}a_{(m+i)}^M\right)
\nno
\end{align}
for $a\in V_{\bar j}$,
$b\in V$,
$m\in \frac{j}{N}+\Z$,
$n\in \frac{1}{N}\Z$,
$r\in \Z$,
where
\begin{align}
V_{\bar {j}}=\{\sigma(a)=(e^{\frac{2\pi \sqrt{-1}}{N}})^{-j}a\}.
\end{align}
The relation
(\ref{eq:Borcherds-id})
is called the
{\em  twisted Borcherds identity}.

By setting
$r=0$ in (\ref{eq:Borcherds-id}),
one obtains 
\begin{align}
 [a^M_{(m)}, b^M_{(n)}]=
\sum_{i=0}^{\infty}
\begin{pmatrix}
 m\\i
\end{pmatrix}
(a_{(i)}b)^M_{(m+n-i)},
\label{eq:commutator-formula}
\end{align}
or equivalently,
\begin{align}
 [Y^M(a,z), Y^M(b,w)]=\sum_{i=0}^{\infty}
Y^M(a_{(i)}b,w)\partial_{w}^{[i]}\delta_j(z-w)
\label{eq:commutator-formula-2}
\end{align}
for $a\in V_{\bar j}$,
where \begin{align*}
\delta_j(z-w)=z^{-j/N}w^{j/N}\delta(z-w)=
\sum_{n\in j/N+ \Z}w^{n}z^{-n-1}.
      \end{align*}
In particular
$Y^M(a,z)$ and $Y^M(b,z)$ 
are mutually local.

The relation (\ref{eq:commutator-formula})
gives \cite{Li96-2}
\begin{align}
\label{eq:2007-12-25-2}
&Y^M(a_{(n)}b, w)\\
&=
\Res_z \sum_{k=0}^{\infty}
\begin{pmatrix}
 -j/N\\ k
\end{pmatrix}
z^{j/N-k}w^{-j/N}
(z-w)^{n+k}[Y^M(a,z),Y^M(b,w)]\nno
\end{align}
for all $n\geq 0$.
The sum in (\ref{eq:2007-12-25-2})
is finite because of the locality.
(In reality 
(\ref{eq:2007-12-25-2})
holds for all $n\in \Z$
in an appropriate sense,
see \cite{Li96-2}).

Set $b=\1$, $r=-2$, $n=0$ 
in (\ref{eq:Borcherds-id}).
It follows that 
\begin{align}
 Y^M(T a,z)=\partial_z Y^M(a,z).
\end{align}

\smallskip

Suppose that $V$ is graded by a Hamiltonian $H$.
A $\sigma$-twisted representation $M$
is called {\em graded}
if there exists an diagonalizable operator $H^M$
on $M$ such that
\begin{align}
 [H^M, a_{(n)}^M]=(T a)^M_{(n+1)}+ (H a)_{(n)}
\label{eq:hamiltonian-aciton}
\end{align} 
for all $a\in V$ and $n\in \frac{1}{N}\Z$.
If $a$ is homogeneous,
\eqref{eq:hamiltonian-aciton} is equivalent to
\begin{align}
 [H^M, a^M_{(n)}]=-(n-\Delta_a+1)a^M_{(n)}.
\label{eq:2008-01-13-14-04}
\end{align}
We set
\begin{align}
 M_d=\{m\in M; H^M m=d m\}
\end{align}
for $d\in \C$.

A {\em positive energy $\sigma$-twisted 
representation}\footnote{A positive energy
representations is also
called an admissible representation
in the literature.}
of $V$
is a  graded $\sigma$-twisted representation $M$
of $V$ such that
there exist a finite set $d_1,\dots,d_r\in \C$
such that
$M_d=0$ unless
$d\in \bigcup_{i}d_i+\Z_{\geq 0}$.
Let $V\adMod_{\sigma}$
be the category of positive energy 
$\sigma$-twisted representations of $V$,
whose morphisms are graded homomorphisms of
$\sigma$-twisted representations. 

An {\em ordinary}
$\sigma$-twisted representation
of $V$ is 
a 
positive energy $\sigma$-twisted representation
of $V$ such that $\dim M_d<\infty$
for all $d$.
Let $V\Mod_{\sigma}$ be the full subcategory of 
$\V\adMod_{\sigma}$
consisting of ordinary $\sigma$-twisted representations.

When $\sigma=\id_V$,
$\sigma$-twisted representations are 
just usual (non-twisted) representations.
We set $V\adMod=V\adMod_{\id_V}$ and $V\Mod=V\Mod_{\id_V}$.

\subsection{$H$-Twisted Zhu Algebras}\label{subsection:Zhu-algbera}
Let $V$ be a vertex algebra graded by a Hamiltonian $H$.
Assume that 
$V_{\Delta}\ne 0$ unless $\Delta\in \frac{1}{N}\Z$.
Then $\sigma_H:=e^{2 \pi i\ad H}:V\ra V$ is an automorphism
of order at most $N$.

If $M$ is a graded $\sigma_H$-twisted representations of 
$V$ then 
the number 
$n-\Delta_a+1$ in  (\ref{eq:2008-01-13-14-04})
is always an integer.
Set $a_{n}^M=a_{(n+\Delta_a-1)}^M$,
so that
\begin{align}
 Y^M(a,z)=\sum_{n\in \Z}a_{n}^M z^{-n-\Delta_a},
\quad [H^M, a_n^M]=-n a_n^M.
\end{align}

 Define the $H$-twisted Zhu algebra  \cite{Zhu96,De-Kac06}
$\Zh_H V$ by
\begin{align}
\Zh_H V= V/V\circ V,
\end{align}
where $V\circ V$ is  the 
span of the vectors \begin{align*}
		     a\circ b:=\sum_{r\geq 0}
\begin{pmatrix}
 \Delta_a\\r
\end{pmatrix}a_{(r-2)}b
		    \end{align*}
with homogeneous  vectors  $a,b\in V$.
The $\Zh_H V$ is an associative algebra
with the multiplication 
\begin{align*}
 a* b=\sum_{r\geq 0}\begin{pmatrix}
		     \Delta_a\\r
		    \end{pmatrix}
a_{(r-1)}b.
\end{align*}

Let $M$ be an object of 
$V \adMod_{\sigma_H}$.
Denote by $V_{\tp}$ the sum of
homogeneous subspace $V_d$ such that
$V_{d'}=0$ for all $d'\in d-\N$.
Then $V_{\tp}$ is naturally a 
module over $\Zh_H V$
by the following action:
\begin{align}
 (a+ V\circ V) m=a^M_{(\Delta_a-1)}m
=a^M_0 m.
\end{align}

\begin{Th}[\cite{Zhu96,De-Kac06}]
The map $M\mapsto M_{\tp}$ gives a bijective correspondence
between simple objects of $V \adMod_{\sigma_H}$
 and irreducible $\Zh_{H}V$-modules.
\end{Th}

The $M$ is said to be {\em almost highest weight}
if 
(1)
$M_{\tp}=M_d$  for some $d$
and
(2)
$M$ is generated by $M_{\tp}$ over $V$.
The $M$ is said be {\em almost co-highest weight}
if 
(1) $M_{\tp}=M_d$  for some $d$ and
(2) $M$ contains  no graded submodule 
intersecting $M_{\tp}$ trivially.
The $M$ is called {\em almost irreducible}
\cite{De-Kac06}
if $M$ is both almost highest weight and 
almost co-highest weight.
Clearly, an almost irreducible module is simple if and only if
$M_{\tp}$ is irreducible over $\Zh_H V$.

\section{Affine $W$-algebras}
%\subsection{}
%In this section we 
%collect the necessary information
%on the affine $W$-algebra $\Wg$
%in the case that $f$ is Richardson.
%The textbooks \cite{Kac90,Kac98,FreBen04}
%and the papers
%\cite{KacRoaWak03,KacWak04,ElaKac05,BruGoo07,De-Kac06,Ara07}
%are our basic references in this section.
\subsection{The Setting}\label{subsection:sl2-triple}
Let $\fing$ be a complex reductive Lie algebra,
$f$ a nilpotent element of $\fing$.
The corresponding  affine $W$-algebra  $\W^k(\fing,f)$
at the level $k\in \C$
is defined by the method of  the quantum BRST reduction.
This method was discovered by Feigin and Frenkel
\cite{FF90} who used it to define 
the $W$-algebra $\Wg$ associated with the 
principal nilpotent elements
$f$.
The most general definition of $\Wg$ was given
by Kac, Roan and Wakimoto
\cite{KacRoaWak03},
and 
the definition 
of $\W^k(\fing,f)$ given in \cite{KacRoaWak03,KacWak04}
involves  another data,
namely a {\em good grading }
of $\fing$ for $f$.
However,
thanks to the results
\cite{BruGoo07}
of  Brundan and Goodwin,
the definition
of $\W^k(\fing,f)$
 does {\em not}
depend on the choice of a good grading for $f$.

Throughout this  paper 
{\em we assume that
$f$ admits a good even grading
}unless otherwise stated,
that is,
there exists 
a $\Z$-grading  
\begin{align}
\sg=\bigoplus_{j\in \Z}\sg_j
\label{eq:good-grading}
\end{align}
of $\sg$
such that $\mathfrak{z}(\fing)\subset \fing_0$,  $f\in \sg_{-1}$, and
$\ad f: \sg_{\leq 0}\ra \sg_{< 0}$ is surjective,
where
 $\sg_{\leq  0}=\bigoplus_{j\leq 0}\sg_j$
and $\sg_{< 0}=\bigoplus_{j<0}\sg_j$. 
The last condition is equivalent to that
$\ad  f: \sg_{>0}\ra \sg_{\geq 0}$ is injective,
where 
$\sg_{\geq 0}=\bigoplus_{j\geq 0}\sg_j$
and $\sg_{>0}=\bigoplus_{j>0}\sg_j$.
By definition
there exists an exact sequence
\begin{align}
0\ra \sg^f\hookrightarrow \sg_{\leq  0}\overset{\ad f}{\ra} \sg_{<0}\ra
 0,
\label{eq:exact-gf}
\end{align}
where $\sg^f$ is the centralizer of $f$ 
in $\sg$.

One can find   
a
$\mathfrak{sl}_2$-triple 
$(e,h,f)$  in $\sg$
such that $e\in \fing_1$,
$h\in \fing_0$, see Lemma 1.1 of \cite{ElaKac05}.
Below {\em we write $h_0$ for $h$}.
Also,
there exists a semisimple element $x_0\in \sg_0$
that defines the $\Z$-grading, i.e.,
\begin{align}
\sg_j=\{a\in \sg;[x_0,a]=j a\}.
\label{eq:grading-g}
\end{align}

We fix a non-degenerate invariant inner product $(~|~)$
on $\fing$
such that $(e|f)=1$.
Set
\begin{align}
\schi=\schi_f=(f|?)\in \fing^*,
\label{eq:character-finite}
\end{align}
and let
$\mathbb{O}_{\schi}$ be the coadjoint orbit of $\schi$,
\begin{align}
 d_{\schi}=\frac{1}{2}\dim \mathbb{O}_{\schi}.
\end{align}
By (\ref{eq:exact-gf})
one has
\begin{align}\label{eq:d=half-of-dim-of-orbit}
\dim \sg_{< 0}=\frac{1}{2}(\dim \sg-\dim \sg^f)=d_{\schi}.
\end{align}

\subsection{Root Data}\label{subsection:Root-Data}
Let 
$\sh$ be a Cartan subalgebra of $\sg_0$
containing
$x_0$ and $h_0$ (see above).
Then $\sh$ is a Cartan subalgebra of $\sg$.
Let $\sroots$ be the set
of roots of $\sg$.
One has 
\begin{align*}
\sroots=\sqcup_{j\in \Z}\sroots_j,
\end{align*}
where $\sroots_j=\{\alpha\in \sroots;
\bra \alpha, x_0\ket=j\}$.
The $\sroots_0$ is the set of roots of the reductive
subalgebra $\sg_0$. 
Let
 $\sroots_{0,+}$ be a set of positive 
roots of $\sg_0$,
$\sroots_{0,-}=-\sroots_{0,+}$.
Then $\sroots_+=\sroots_{0,+}\sqcup \sroots_{>0}$
is a set of positive roots of $\sg$,
where $\sroots_{>0}=\sqcup_{j>0}\sroots_j$.
Likewise,
$\sroots_-=\sroots_{0,-}\sqcup \sroots_{<0}$
is a set of negative roots of $\sg$,
where 
$\sroots_{<0}=\sqcup_{j<0}\sroots_j$.
Let
$\sg_0=\sn_{0,-}\+ \sh\+ \sn_{0,+}$
and 
$\sg=\sn_-\+ \sh\+ \sn_+$
be the corresponding
 triangular decompositions of $\sg_0$ and $\sg$,
respectively.

Let $\bar \rho$ be the half sum of positive roots of $\fing$.

Let $\bar Q$,
$\bar Q\che$,
$\sP$
and $\sP\che$
be the root lattice,
the coroot lattice,
the weight lattice and the coweight lattice 
of $\fing$,
respectively.
Denote by $\sW$ the Weyl group of $\fing$.

We fix an anti-Lie algebra involution
$\sg\ni x\mapsto  x^t\in \sg$
such that 
$e^t=f$ and $h^t=h$ (for all $h\in \sh$). 

Set $\sI=\{1,\dots, \rank \sg\}$.
Let $\{J_a; a\in \sI\sqcup \sroots_{\pm}\}$
be a basis of $\fing$
such that $J_{\alpha}$ with $\alpha\in \sroots$
is a root vector of root $\alpha$
and $\{J_i;i \in \sI\}$ is a basis of $\sh$.
Denote by $c_{a,b}^d$ the corresponding structure constant.

\subsection{Kac-Moody Lie Algebras}
\label{subsection:Kac-Moody}

Let $\bg$ be the 
Kac-Moody affinization of
 $\sg$:
\begin{align}
 \bg=\sg[t,t\inv]\+ \C K\+ \C D,
\end{align}
where $\sg[t,t\inv]=\sg\* \C[t,t\inv]$.
The commutation
relations
are given by
\begin{align}
 [X(m),Y(n)]=[X,Y](m+n)+m\delta_{m+n,0}(X|Y)K,\\
[D,X(m)]=m X(m),\quad [K,\bg]=0
\end{align}
for $X,Y\in \affg$, $m,n\in \Z$,
where
$X(m)=X\* t^n$.
The subalgebra $\fing\* \C\subset \affg$
is naturally identified with $\sg$.
%Also, we set
%\begin{align}
% &\sg[t]=\sg\*\C[t],\quad
%\sg[t]t=\sg\* \C[t]t,\\
%\end{align}

The 
form $(~|~)$
is 
extended from $\sg$ to 
the invariant symmetric bilinear  of $\affg$
as follows:
\begin{align*}
 (X(m)|Y(n))=(X|Y)\delta_{m+n,0},
\quad (D| K)=1,\\
(X(m)|D)=(X(m)|K)=(D|D)=(K|K)=0.
\end{align*}

We fix the 
triangular decomposition
$\affg=\affn_-\+\affh\+\affn_+$
as usual:
 \begin{align}
&\h=\sh\+\C K\+\C D,\label{eq:Cartan-affine}\\
&\affn_-=\snn[t\inv]\+ \sh[t\inv]t\inv
\+ \snp[t\inv]t\inv,\\
&\text{$\affn_+=\sn_-[t]t\+ \sh[t]t
\+ \snp[t]$.}
\end{align}

Let
$
\dual{\h}=\dual{\sh}\+\C \Lam_0\+\C \delta
$
be the  dual of $\h$.
Here,
$\Lam_0$ and
$\delta$ are dual elements of $K$ and $D$,
respectively. For $\lam\in\dual{\h}$,
the number
$\bra \lam,K\ket$ is called the
{\em level of $\lam$}.
%Let
%$\dual{\h}_{\kappa}$
%denote the
%set of the
%weights of level
%$k$:
%\begin{align}
%\dual{\h}_{k}
%\teigi\{\lam\in \dual{\h};\bra \lam,K\ket=k\}.
%\end{align}
%$$\dual{\h}_{k}=\{\lam\in\dual{\h}
%;  \bra \lam+\rho,K\ket=k\}.$$

Let
$\bar{\lam}$ be the restriction of $\lam\in \dual{\h}$ to $\dual{\sh}$.
We refer to $\slam$ as the {\em classical part}
of $\lam$.

\smallskip

Let $\roots$ be the set of roots of $\affg$,
$\roots_+$ the set of positive roots,
$\roots_-=-\roots_+$.
Then,
$\roots=\rroots\sqcup \iroots$,
where
$\rroots$
is
the set of real roots
and
$\iroots$ is the set of imaginary roots.
Let
%$\Pi$ be the standard basis of $\rroots$ (?),
$\rroots_{\pm}=\rroots\cap \roots_{\pm}$
and
$\iroots_{\pm}=\iroots\cap \roots_{\pm}$.
One has
\begin{align*}
 \rroots_+=\{ \alpha+n\delta;  \alpha\in \sroots_+,\ n\geq 0\}\sqcup
\{ -\alpha+n\delta;  \alpha\in \sroots_+,\ n\geq 1\}.
\end{align*}
Let $Q$ be the root lattice,
$Q_+=\sum_{\alpha\in \roots_+
}\Z_{\geq 0} \alpha\subset
Q$.
We define a partial 
ordering $\mu\leq  \lam$  on $\dual{\bh}$
by
$\lam-\mu\in Q_+$.

\smallskip 
Let
$W\subset GL(\dual{\h})$ be the Weyl group of $\affg$
generated by the reflections $s_{\alpha}$
with $\alpha\in \rroots$,
where
$s_{\alpha}(\lam)=\lam-\bra \lam,\alpha\che\ket\alpha$
for $\lam\in \dual{\bh}$.
%Let $\eW$ be
%the extended Weyl group of $\affg$:
%$\eW=\sW\ltimes \sP\che$.
%We write $t_{\mu}$
%for the element of 
%of $\eW$   corresponding to
%$\mu\in \sP\che$:
%\begin{align*}
%t_{\mu}(\lam)=\lam+\bra\lam,K\ket\mu-
%\left(\bra\lam,\mu\ket+\frac{1}{2}
%|\mu|^2\bra\lam,K\ket
%\right)\delta \quad \text{for $\lam\in \dual{\h}$.}
%\end{align*}
The dot action of
$W$
on $\dual{\h}$ is defined by $w\circ \lam= w(\lam+\rho)-\rho$,
where
$\rho=\bar{\rho}+h\che
\Lam_0\in\dual{\h}$
and $h\che$ is the dual Coxeter number of $\fing$.

For $\lam\in \dual{\bh}$,
let
\begin{align}
& \roots(\lam)= \{\alpha\in \rroots; 
\bra \lam+\rho,\alpha\che\ket\in \Z\},\\
&\roots_+(\lam)= \roots(\lam)\cap \roots_+,\\
&W(\lam)=\bra s_{\alpha};  \alpha\in \roots(\lam)\ket
\subset W.
\end{align}
One knows that $W(\lam)$ is a Coxeter subgroup of $W$,
and $W(\lam)$ is called the {\em integral Weyl group} of $\lam\in \dual{\bh}$.
Let
$\ell_{\lam}:W(\lam)\ra \Z_{\geq 0}$
be its length function.

\smallskip

For an $\bh$-module 
$M$ let $M^{\lam}$ 
be the weight space of weight $\lam\in \dual{\bh}$:
\begin{align*}
M^{\lam}=\{m\in M; a m=\lam(a)m~ \forall a \in \bh\}.
\end{align*}
We say $M$ admits a weight space 
decomposition if 
$M=\bigoplus_{\lam}M^{\lam}$
and $\dim M^{\lam}<\infty$ for all $\lam$.
In this case
we define the graded dual $M^*$
of $M$ by
\begin{align}\label{eq:graded-dual}
M^*=\bigoplus_{\lam}\Hom_{\C}(M^{\lam},\C)
\subset \Hom_{\C}(M,\C).
\end{align}
Also, we set\footnote{This differs from the notation in \cite{Ara07}. }
\begin{align}
& M_d=\{m\in M; -D m=d m\},\\
&\sfD(M)=\bigoplus_{d\in \C}\Hom_{\C}(M_d,\C)
\end{align}
for a semisimple $\C D$-module $M$.
Note that
if $M$ is a $\affg$-module
then $M_d$ is a $\fing$-submodule of $M$ for any $d$.

\begin{Lem}\label{Lem:easy-lemma}
 Let $M$ be a $\bh$-module that admits a weight space decomposition.
Suppose that $M_d$ is finite-dimensional for all $d$.
Then $\sfD(M)=M^*$.
\end{Lem}
\subsection{Universal Affine Vertex Algebras}
For $k\in \C$
define
the $\bg$-module $\Vkg$ by
\begin{align}
 \text{$\Vkg\teigi
U(\affg)\*_{U(\sg[t]\+\C K\+
\C \Dg)}\C_k
$,
}
\end{align}
where
$\C_k$
is  the one-dimensional representation
of
$\sg[t]\+\C K\+
\C \Dg$
on which
$\sg[t]\+
\C \Dg$
acts trivially
and $K$ acts as the multiplication by $k$.

Define a field $J(z)$ on $\Vkg$ for $J\in \fing$
by
\begin{align}
J(z)=\sum_{n\in \Z}J(n)z^{-n-1}.
\end{align}

There is a unique
vertex algebra
structure on
$V^k(\sg)$
such that 
$\1=1\* 1\in V^k(\fing)$ is the vacuum vector 
and
$\{J(z); J\in \fing\}$ is a set of generating fields.
The   vertex algebra
$\Vkg$ is
 called the 
{\em universal affine vertex algebra}
associated with $\sg$
at level $k$.

\subsection{Clifford Vertex Algebras}
Set 
\begin{align}
&L{\sg_{>0}}=\sg_{>0}\* \C[t,t\inv],
\quad L{\sg_{<0}}=\sg_{<0}\* \C[t,t\inv].
\end{align}
They
are nilpotent subalgebras of $\affg$.

Let
$\Cl$ be
  the {Clifford algebra}
associated with
$L{\sg_{<0}}\+L{\sg_{>0}}
$
and the restriction  of $(~|~)$ to $L{\sg_{<0}}\+ L{\sg_{>0}}$.
%As $\sCl$,
The superalgebra $\bCl$
has the following generators and relations:
\begin{align*}
&\text{generators: }\psi_{\alpha}(n)
&(\alpha\in
 \sroots_{\ne 0},
n\in \Z),\\
&\text{relations: }\psi_{\alpha}(n)\text{ is odd},\\
&\qquad  \quad \quad
[\psi_{\alpha}(m),\psi_{\beta}(n)]=\delta_{\alpha+\beta,0}
\delta_{m+n,0}
&
(\alpha,\beta\in \sroots_{\ne 0},
m,n\in \Z).
\end{align*}
Here $\sroots_{\ne 0}=\sroots_{<0}\sqcup \sroots_{>0}$.

The algebra 
$\Cl$   contains
the Grassmann algebras
$\bigwedge(L{\sg_{<0}})$ 
and $\bigwedge(L{\sg_{>0}})$ 
 as its subalgebras;
$\bigwedge(L{\sg_{<0}})
=\bra \psi_{\alpha}(n);  \alpha\in \sroots_{<0},n\in \Z\ket$,
$\bigwedge(L{\sg_{>0}})
=\bra \psi_{\alpha}(n);  \alpha\in \sroots_{>0},n\in \Z\ket$.
One has
\begin{align}
 \Cl=\bigwedge(L{\sg_{>0}})\*\bigwedge(L{\sg_{<0}})
\label{eq:2006-07-31-14-06}
\end{align}
 as  linear spaces.

In view of \eqref{eq:2006-07-31-14-06},
the adjoint action of $\bh$
on $L{\sg_{<0}}\+ L{\sg_{>0}}$ induces an 
 action of $\h$ on $\Cl$:
$\bCl=\bigoplus_{\lam\in Q}\bCl^{\lam}$.

Let $\Lamsemi{\bullet}(L{\sg_{>0}})$  be the irreducible
representation of $\Cl$ generated by the vector $\1$
such that
\begin{align}
 \text{$\psi_{\alpha}(n)\1=0$
\quad 
if
$\alpha+n\delta\in \prroots$.
}
\end{align}
Then
$
\semiLamp{\bullet}= 
\bigwedge(L{\sg_{<0}}\cap \affn_-)\*\bigwedge(L{\sg_{>0}}\cap \affn_-)
$
as  linear spaces.
We regard $\Lamsemi{\bullet}(L{\sg_{>0}})$
 as an $\bh$-module
under this identification:
\begin{align*}
 \Lamsemi{\bullet}(L{\sg_{>0}})=\bigoplus_{\lam\in -Q_+}
\Lamsemi{\bullet}(L{\sg_{>0}})^{\lam}.
\end{align*}

The space $\semiLamp{\bullet}$ is $\Z$-graded
by {\em charge}:
\begin{align}
\semiLamp{\bullet}=\bigoplus_{i\in \Z}\semiLamp{i}.
\end{align}
The charge 
of
$\1$,
$\psi_{\alpha}(n)$ and
$\psi_{-\alpha}(n)$ 
with $\alpha\in \sroots_{>0}$
and $n\in \Z$
are $0$,
 $-1$ and $1$,
respectively.
The $\Cl$-module
$\semiLamp{\bullet}$ is 
often called the 
{\em space of  semi-infinite forms} on $L\fing_{>0}$.

There is a unique vertex (super)algebra
structure on
%The space of semi-infinite forms
$\semiLamp{\bullet}$
such that
$\1$ is the vacuum vector,
and  
\begin{align}
& %Y(\psi_{\alpha}(-1)\1,z)=
\psi_{\alpha}(z)\teigi \sum_{n\in \Z}\psi_{\alpha}(n)z^{-n-1}
\quad 
\text{with $\alpha\in \sroots_{>0}$,
}\\
& %Y(\psi_{\alpha}(0)\1,z)=
\psi_{\alpha}(z)\teigi \sum_{n\in \Z}\psi_{\alpha}(n)z^{-n}
\quad \text{with $\alpha\in \sroots_{<0}$
}
\end{align}
are generating fields.
The vertex algebra $\semiLamp{\bullet}$
is also  called the {\em Clifford vertex algebra}
associated with $L\sg_{>0}$.

\subsection{The $W$-Algebra $\Wg$}
Because 
both $\Vkg$ and $\semiLamp{\bullet}$ 
are vertex algebras,
the tensor product
%Define the vertex algebra $\bCg$ by
\begin{align}
\bCg
=
\Vkg \* \semiLamp{\bullet}
\label{eq:2006-05-17-13-15}
\end{align}
is also a vertex algebra.
Set $\mathcal{C}^i=V^k(\fing)\* \semiLamp{i}$,
so that 
\begin{align}
\bCg=\bigoplus_{i\in \Z}\mathcal{C}^i.
\end{align}

Let $Q(z)$ be the odd field on $\bCg$
defined by
\begin{align}
 Q(z)
= Q^{\st}(z)+\chi(z),
\end{align}
where
\begin{align*}
& Q^{\st}(z)
=
\sum_{\alpha\in \sroots_{>0}}
J_{\alpha}(z)\psi_{-\alpha}(z)-\frac{1}{2}
\sum_{\alpha,\beta,\gamma\in \sroots_{>0}}c_{\alpha,\beta}^{\gamma}
\psi_{-\alpha}(z)\psi_{-\beta}(z)\psi_{\gamma}(z), %\nonumber
 \\
 &\chi(z)
=
\sum_{\alpha\in \sroots_{>0}}\schi(J_{\alpha})
\psi_{-\alpha}(z)
\end{align*}
($\schi$ is defined in (\ref{eq:character-finite})).
One has %the operator production expansion
\begin{align}\label{eq:ope=0}
 [Q(z),Q(w)]=0,
\end{align}
and therefore,
\begin{align}\label{eq:ope=0,2}
Q_{(0)}^2=0
\end{align}
because $Q(z)$ is odd.
(Recall that $Q_{(0)}=\Res_zQ(z)$, see \S \ref{subsection:fields}.)

Since $\affQp \mathcal{C}^i\subset \mathcal{C}^{i+1}$,
 $(\bCg,\affQp)$
is a cochain complex.

\begin{Def}%[\cite{FF90,BoeTji94,KacRoaWak03}]
The universal affine $W$-algebra $\Wg$
associate with $(\sg,f)$ at level $k$
is 
the zeroth cohomology of the complex $(\bCg,\affQp)$:
\begin{align}\label{eq:def-of-affine-W}
 \Wg=H^0(\bCg, \affQp).
\end{align}
\end{Def}
 The $W$-algebra
 $\Wg$ inherits the vertex algebra structure from $\bCg$.

\subsection{The Hamiltonian of $\Wg$}
Set
\begin{align}
\label{eq:Hamitonian}
 H=-D-\frac{1}{2}h_0,
\end{align}
where
$h_0$ is the element in the $\mathfrak{sl}_2$-triple
$\{e,h_0,f\}$ fixed in \S \ref{subsection:sl2-triple}.
The right-hand-side is considered as an element
of $\affh$
which acts diagonally on 
the complex $\bCg=\Vkg\* \semiLamp{\bullet}$.

One knows that  $H$ defines a Hamiltonian 
on $\bCg$
(cf.\ \S 4.9 of \cite{Kac98}).
\begin{Lem}\label{Lem:Hamiltonian-commutes}
One has
$[H, \affQp]=0$,
\end{Lem} 
\begin{proof}
 Obviously $[H, Q_{(0)}^{\st}]=0$.
Also,
\begin{align}
 \text{$\alpha(h_0)=2$
for all $\alpha$ such that $\schi(J_{\alpha})\ne 0$.
}\label{eq:h-vs-x}
\end{align}This gives $[H, \chi_{(0)}]=0$.
\end{proof}
From Lemma  \ref{Lem:Hamiltonian-commutes}
it follows that $H$ defines a Hamiltonian of $\Wg$.
One has
\begin{align}
 \Wg&=\bigoplus_{\Delta\in \frac{1}{2}\Z}\Wg_{\Delta},\nno
\\ &\Wg_{\Delta}=\{a\in \Wg; Ha =\Delta a\}. 
\end{align} 
\subsection{Generating Fields of $\Wg$}
\label{subsection:generating-field-of-W}
Set
\begin{align}
\wJ_{a}(z)=\sum_{n\in \Z}\wJ_{a}(n)z^{-n-1}
=J_a(z)-\sum_{\beta,\gamma\in \roots_{>0}}c_{\alpha,\beta}^{\gamma}
:\psi_{-\beta}(z)\psi_{\gamma}(z):
\end{align}
for $a\in \sI\sqcup \sroots$.
One has \cite[(2.5)]{KacWak04} on $\bCg$
\begin{align}
 &[\wJ_a(m),\wJ_b(n)]\label{eq:commutation-relation-wJ}
\\
&=\sum_{d}c_{a,b}^d \wJ_d(m+n)
+\left((k+h\che)(a|b)-\frac{1}{2}\kappa_{\sg_0}(a,b)\right)
m\delta_{m+n,0}\id \nno,\\
&[\wJ_a(m),
 \psi_{\alpha}(n)]=\sum_{d}c_{a,\alpha}^{\beta}\psi_{\beta}(m+n)
\label{eq:commutation-relation-wJ-2}
\end{align} 
provided  that either
$a,b\in \sroots_{\geq 0}\sqcup \sI$
and $\alpha\in \sroots_{>0}$,
or $a,b\in \sroots_{\leq 0}\sqcup \sI$
and $\alpha\in \sroots_{<0}$,
where $\kappa_{\sg_0}(a,b)$ is the 
Killing form of $\sg_0$.

Let $\bCg_+$
be the vertex subalgebra of
$\bCg$ generated by 
the fields $\wJ_a(z)$ and $\psi_{\beta}(z)$ with 
$a\in \sI\sqcup \sroots_{\leq 0}$ and $\beta\in \sroots_{<0}$.
By (\ref{eq:commutation-relation-wJ})
and (\ref{eq:commutation-relation-wJ-2}),
$\bCg_+$ is spanned by the vectors 
\begin{align*}
\wJ_{a_1}(m_1)\dots \wJ_{a_r}(m_r)\psi_{\beta_1}(n_1)\dots
\psi_{\beta_s}(n_s)\1
\end{align*}
 with 
$a_i\in \sI\sqcup \sroots_{\leq 0}$,
$\beta_i\in \sroots_{<0}$,
$m_i,n_i\in \Z$.
%Here $\1=1\* \1\in \bCg$.

Similarly
let $\bCg_-$
be the vertex subalgebra of
$\bCg$ generated by 
the fields $\wJ_{\alpha}(z)$ and $\psi_{\beta}(z)$ with 
$\alpha, \beta\in \sroots_{>0}$.
Then
%by (\ref{eq:commutation-relation-wJ})
$\bCg_-$ is spanned by the vectors 
\begin{align*}
\wJ_{\alpha_1}(m_1)\dots \wJ_{\alpha_r}(m_r)\psi_{\beta_1}(n_1)\dots
\psi_{\beta_s}(n_s)\1
\end{align*} with 
$a_i\in \sroots_{>0}$,
$\beta_i\in \sroots_{>0}$,
$m_i,n_i\in \Z$.

One has the linear isomorphism
\begin{align}
 \bCg\cong \bCg_-\* \bCg_+.
\label{eq:tensorproduct-decomposition}
\end{align} 
Moreover  
it was shown  \cite{KacWak04} (cf.\ \cite{BoeTji94,FreBen04})
that
both $\bCg_{\pm}$ are subcomplexes of 
$\bCg$, 
and that 
\begin{align*}
H^i(\bCg_-)=\begin{cases}
	      \C &(i=0)\\
0&(i\ne 0)
	     \end{cases}.
\end{align*}
Therefore by the K\"{u}nneth theorem
\begin{align}
 H^{\bullet}(\bCg)=H^{\bullet}(\bCg_+).
\label{eq:2007-12-26-1}
\end{align}
It follows that  we may identify $\Wg$
with the vertex subalgebra $H^0(\bCg_+)$
of $ \bCg$
(Note that the cohomological gradation takes only non-negative values on $\bCg_{+}$.
):
\begin{align}
 \Wg=H^0(\bCg_+)\subset \bCg.
\label{eq:W-as-sub}
\end{align}

Let $\sg^f_{\aff}=\sg^f \* \C[t,t\inv] \+ \C \1$
be the central extension 
of the Lie algebra
$\sg^f\* \C[t,t\inv]$
with respect to the 2-cocycle $\phi_k$,
defined by
$\phi_k(a,b)=(k+ h\che)
	(a|b)-\frac{1}{2}\kappa_{\sg_0}(a,b)$.
Set 
\begin{align*}
V^{\phi_k}(\sg^f)=U(\sg^f_{\aff})\*_{U(\sg^f\* \C[t]\+ \C \1)}\C
\end{align*}
where $\C$ is the  
$\sg^f\* \C[t]\+ \C \1$-module on which
$\sg^f\* \C[t]$ acts trivially and $\1$ acts as $1$.

By (\ref{eq:commutation-relation-wJ})
one can regard 
$V^{\phi_k}(\sg^f)$
as a vertex subalgebra of $\Vkg$.

\begin{Th}[\cite{KacWak04}]
\label{Th:KW} 
For any $k\in \C$ one has the following.

 \begin{enumerate}
  \item It holds that
$H^i(\bCg_+)=0$ for all $i\ne 0$.
Therefore
$H^i(\bCg)=0$ for all $i\ne 0$.
\item There exists an exhaustive,
separated  filtration $\{F_p \Wg\}$
of the vertex algebra $\Wg$
such that
\begin{align*}
 \gr^F \Wg\cong  V^{\phi_k}(\sg^f)
\end{align*}
as graded vertex algebras.
 \end{enumerate}
\end{Th}

\begin{Rem}
{\rm
 The filtration in Theorem \ref{Th:KW}
arises from the spectral sequence 
%considered in \cite{KacRoaWak03}
associated with  the filtration of 
$\bCg_+$ defined by \begin{align*}
F_p \mathcal{C}_+^n=
\bigoplus\limits_{\bra \lam,x_0\ket\geq  p-n}(\mathcal{C}_+^n)^{\lam}
		    \end{align*}
 (cf.\ \S 4 of \cite{Ara07}). }
\end{Rem}
Because $\sg^f$
is preserved by the adjoint action
of $x_0$ and $h_0$,
there exists a basis
 $\{u_j;j=1,\dots, \dim \sg^f\}$
of $\sg^f$
consisting of simultaneous
eigenvectors of $\ad x_0$
and $\ad h_0$.
Let $d_j\in \frac{1}{2}\Z_{\geq 0}$ be the 
half of the eigenvalue of
$\ad h_0$ on $u_j$:
\begin{align}
 [h_0, u_j]=-2 d_j u_j.
\end{align}

By Theorem \ref{Th:KW} there exist
homogeneous elements $\vecW{j}$
of $\Wg$
 with $j=1,\dots, \dim \sg^f$
whose symbols
%\footnote{That is,
%$\vecW{j}\equiv u_j(-1)\1\pmod{\sum\limits_{\mu\atop
%\mu(x_0)> -d_j'} (\mathcal{C}^0)^{\mu}}$
%if $u_j\in \sg^f\cap \sg_{-d_j'}$.} 
are $u_j(-1)\1$,
and $\Wg$ is (strongly \cite{Kac98})
generated by the fields
\begin{align}
 \vecW{j}(z)=Y(\vecW{j},z)
\end{align}
in $\bCg$.
The vector $\vecW{j}$ has the conformal weight 
$1+ d_j$.
Thus it follows that
$\Wg$ is positively graded:
\begin{align}
 \Wg=\bigoplus_{\Delta\in \frac{1}{2}\Z_{\geq 0}}
\Wg_{\Delta},\quad
\Wg_{0}=\C 1.
\end{align}

\section{Ramond Twisted Representation of Affine $W$-Algebras}
\subsection{Ramond Twisted Representations of $\Wg$}
Let $\sigma_R: \bCg\ra \bCg$ be the automorphism
of order $\leq 2$
 defined by 
\begin{align}
 \sigma_R=e^{\pi i\ad h_0}.
\end{align}
By (\ref{eq:h-vs-x}),
$\sigma_R$ fixes 
the vector $Q=Q_{(-1)}\1$.
Therefore \cite{KacWak05} %$\sigma_R \cdot \affQp \cdot \sigma_R\inv=\affQp$,
%and hence 
$\sigma_R$ defines an automorphism of $\Wg$.

A $\sigma_R$-twisted representation of 
$\Wg$ is called a 
{\em Ramond twisted  representation}
of
$\Wg$. 

Note that $\sigma_R=\sigma_H$ (see \S \ref{subsection:Zhu-algbera}
and (\ref{eq:Hamitonian})).
Therefore Ramond twisted representations
are exactly the $H$-twisted representations.
%defined in  \S \ref{subsection:Zhu-algbera}.
\begin{Rem}
{\rm  If the nilpotent element $f$ is even 
then $\sigma_R$ is trivial.
In this case a Ramond twisted representations are  usual 
(non-twisted)
representations.
}\end{Rem}

%Let $M$ be a Ramond twisted representation of $\Wg$.
%Then the filed $Y^M(\vecW{j},z)$ is expanded as
%\begin{align}
% Y^M(\vecW{j},z)=\sum_{n\in d_j+\Z}(\vecW{j})^M_{(n)}z^{-n-1}.
%\end{align}
%Suppose $M$ is graded by $H^M$.

\begin{Pro}\label{Pro:2007-12-18-03}
 Let $M$ be a $\sigma_R$-twisted  representation of
$\bCg$.
Then
the space 
\begin{align*}
\frac{ \ker((Q)^M_{(0)}: M\ra M)}
{ \im((Q)^M_{(0)}: M\ra M)}
\end{align*}
is naturally a Ramond twisted representation of $\Wg$.
\end{Pro}
\begin{proof}
%Here of course
%$(Q)^M_{(0)}=\Res_z Y^M((Q)_{(-1)}\1,z)$.
%
 By  
(\ref{eq:commutator-formula})
and (\ref{eq:ope=0,2}),
the square of $(Q)^M_{(0)}$ is equal to zero.
Therefore the 
above space is well-defined.
The rest also follows from (\ref{eq:commutator-formula}).
\end{proof}

\subsection{$\sigma_R$-Twisted Representations of $\bCg$}
Set
\begin{align*}
\sg_{j}^{\Dyn}=\{x\in \sg; [h_0,x]=2j x\}.
\end{align*}
Then 
$\fing=\bigoplus_{j\in \frac{1}{2}\Z}\fing_j^{\Dyn}$
gives a good grading for $f$,
called the {\em Dynkin grading} \cite{KacRoaWak03}.

Let 
\begin{align*}
\affg^R=\bigoplus_{j\in \frac{1}{2}\Z}\fing^{\Dyn}_{\bar j}\* \C t^j
\+ \C K \+ \C D
\end{align*} be the $\sigma_R$-twisted
affine Lie algebra \cite{Kac90},
where 
$\fing_{\bar{j}}^{\Dyn}=\bigoplus\limits_{r\in \frac{1}{2}\Z
\atop r\equiv j \pmod{\Z}}\fing_r^{\Dyn}$,
and the commutation relations are given by the same formula as $\affg$.

We write $J(n)^R$ for $J\* t^n\in \affg^R$.
Also, to avoid confusion we write $K^R$ and $D^R$ 
for $K$ and $D$ in $\affg^R$,
respectively.
\begin{Lem}\label{Lem:2007-12-18-1}
Let $M$ be a vector space.
Defining a $\sigma_R$-twisted $V^k(\affg)$-module structure 
on $M$ is  equivalent to 
defining a $\affg^R$-module structure 
on $M$ of level $k$
%(that is, $K$ acts as a multiplication by $k$)
such that $J(n)^R m=0$ for all $m\in M$
and $n\gg 0$.
\end{Lem}
\begin{proof}
By (\ref{eq:commutator-formula-2}),
given a $\sigma_R$-twisted module structure on $M$
one has 
\begin{align}
\label{eq:07-12-18-1}
& [Y^M(J(-1)\1,z), Y^M(J'(-1)\1,w)]\\
&=Y^M([J,J'](-1)\1,w)\delta_j (z-w)
+k (J|J')\id_M\partial_w \delta_j(z-w)
\nno 
\end{align}
for $J\in \fing_{\bar j}^{\Dyn}$, $J'\in \fing$.
%From this,
It follows that 
the correspondence 
$J(n)^R\mapsto (J(-1)\1)_{(n)}^M$
define a representation of 
$\affg^R$ on $M$
of level $k$ with $J(n)^R m=0$ for $m\in M$
and $n\gg 0$. 

Conversely,
suppose that we are given a
$\affg^R$-module structure on $M$ of level $k$ 
such that 
$J(n)^R m=0$ for $m\in M$
and $n\gg 0$.
Define a $2$-twisted filed 
$J(z)^R$ on $M$ by
\begin{align}
 J(z)^R=\sum_{n\in j+\Z}J(n)^R z^{-n-1}
\quad \text{for $J\in \fing^{\Dyn}_{\bar j}$}.
\end{align}
These fields satisfy the same formula  as (\ref{eq:07-12-18-1}):
\begin{align}
 [J(z)^R, J'(w)^R]=[J,J'](w)^R\delta_j (z-w)
+ k(J|J')\id_M  \partial_w  \delta_j (z-w).
\label{eq:2007-12-18-02}
\end{align}
for $J\in \fing_{\bar j}^{\Dyn}$ and $J'\in \fing$.

Let
$V%\subset \{\text{$2$-twisted field on $M$}\}
$ 
be a vertex algebra 
generated by $J(z)^R$ with $J\in \fing$
in the space of $2$-twisted fields on $M$
in the sense of Li \cite{Li96-2}.
By (\ref{eq:2007-12-18-02})
%one has
%\begin{align*}
%&J(z)^R_{(0)}J'(z)^R= [J,J'](z)^R,
%\quad J(z)^R_{(1)}J'(z)^R=k(J|J')\id_M
%\end{align*}
it follows that
the correspondence $J(-1)\1\mapsto J(z)^R$
defines a vertex algebra homomorphism 
from $V^k(\fing)$ to $V$ (cf.\ (\ref{eq:2007-12-25-2})).
Thanks to  Proposition 3.17 of \cite{Li96-2},
this completes the proof.

\end{proof}

Let $\affCl^R$ be the superalgebra generated by the odd fields
$\psi_{\alpha}(n)^R$ ($\alpha\in \sroots_{\ne 0}$,
$n\in \alpha(h_0)/2+\Z$)
with the relations
$[\psi_{\alpha}(m)^R,
\psi_{\beta}(n)^R]=\delta_{m+n,0}\delta_{\alpha+\beta,0}.$

The proof of the following assertion is the same as
that of 
Lemma \ref{Lem:2007-12-18-1}.
\begin{Lem}\label{Lem:2007-12-18-2}
Let $M$ be a $\bCl^R$-module such that 
$\psi_{\alpha}(n)^R m=0$
for all $m\in M$,
$\alpha\in \sroots_{\ne 0}$
and $n\gg 0$.
Then the formulas
\begin{align*}
& Y^M(\psi_{\alpha}(-1)\1,z)=\psi_{\alpha}(z)^R
=\sum_{n\in \alpha(h_0)/2+\Z} \psi_{\alpha}(n)^R z^{-n-1}
\quad (\alpha\in \sroots_{>0}),\\
& Y^M(\psi_{\alpha}(0)\1,z)=\psi_{\alpha}(z)^R
=\sum_{n\in \alpha(h_0)/2+\Z} \psi_{\alpha}(n)^R z^{-n}
\quad (\alpha\in \sroots_{<0})
\end{align*}
defines a $\sigma_R$-twisted  $\semiLamp{\bullet}$-module structure 
on $M$.
\end{Lem}
Set $U_k(\affg^R)=U(\affg^R)/\bra K-k1\ket$.
Let $M$ be a $U_k(\affg^R)\* \bCl^R$-module such that
$J(n)^R m=\psi_{\alpha}(n )^R m=0$
for $n\gg 0$,
$m\in M$, $J\in \fing$ and $\alpha\in \sroots_{\ne 0}$.
Then by Lemmas \ref{Lem:2007-12-18-1}
and \ref{Lem:2007-12-18-2}, 
$M$ can be naturally considered as a $\sigma_R$-twisted
representation of $\bCg$.
By Proposition \ref{Pro:2007-12-18-03},
the space
$\ker (Q)^M_{(0)}/ \im (Q)^M_{(0)} $
is a Ramond twisted representation
of $\Wg$. 
One has
\begin{align}
 (Q)_{(0)}^M=(Q^{\st})_{(0)}^M + \chi_{(0)}^M,
\end{align}
where 
$(Q^{\st})_{(0)}^M $ and $ \chi_{(0)}^M$
are explicitly expressed as follows:
\begin{align*}
& (Q^{\st})_{(0)}^M=
\sum_{\alpha\in \sroots_{>0}
\atop
n\in \alpha(h_0)/2
+\Z }J_{\alpha}(n)^M\psi_{-\alpha}(-n)^M
\\&\qquad\qquad-\frac{1}{2}\sum_{\alpha,\beta,\gamma\in \sroots_{>0}
\atop k
\in \alpha(h_0)/2 +\Z,
l\in \beta(h_0)/2+\Z}
c_{\alpha,\beta}^{\gamma}\psi_{-\alpha}(-k)^M\psi_{-\beta}(-l)^M\psi_{\gamma}(k+l)^M,
\\& \chi_{(0)}^M=\sum_{\alpha\in \sroots_{>0}}\chi(J_{\alpha})\psi_{-\alpha}(1)^M.
\end{align*}

\subsection{Identification with Non-Twisted Representations}
The superalgebra 
$U(\affg^R)\* \bCl^R$ is isomorphic to
$U(\affg)\* \bCl$ \cite{KacWak}:
the isomorphism is given by:
\begin{align*}
 \widehat t_{-\frac{1}{2}h_0}:
& J_{\alpha}(n)^R\mapsto J_{\alpha}(n+\alpha(h_0)/2) 
\quad& (\alpha\in \sroots),
\\
& J_i(n)^R\mapsto J_i(n)\quad &(i\in \sI, \ n\ne 0),\\
& J_i(0)^R \mapsto J_i(0)
+\frac{1}{2}(h_0|J_i)K
\\
&K^R\mapsto K, \\
&D^R \mapsto D-\frac{1}{2}h_0(0),\\
& \psi_{\alpha}(n)^R\mapsto \psi_{\alpha}(n+\alpha(h_0)/2) 
\quad& (\alpha\in \sroots_{\ne 0}, n\in \Z),
\end{align*}

Set $U_k(\affg)=U(\affg)/\bra K-k \ket$ for $k\in \C$.
Let $\widehat{w}_0\in \Aut (U_k(\affg)\* \bCl)$
be a lift of the longest element  
$w_0$ of the Weyl group $\finW$.
Set
\begin{align}
 \widehat{y}_0=\widehat{w}_0\widehat{t}_{-\frac{1}{2}h_0}.
\end{align}
Then
$\widehat{y}_0$ defines an isomorphism 
$U_k(\affg^R)\* \bCl^R \isomap U_k(\affg)\* \bCl$.

Let $M$ be a (non-twisted) positive energy representation of
$V^k(\fing)$.
Then 
the space $M\* \semiLamp{\bullet}$ 
can be regarded as a $\sigma_R$-twisted representation of $\bCl$,
by the action
\begin{align}
u\cdot m=\widehat{y}_0(u)m  \label{eq:twisted-action}
\end{align}
for $m\in M\* \semiLamp{\bullet}$ and $u\in U_k(\affg^R)\* \bCl^R$.
Note that in this case
the differential
$(Q)_{(0)}^{M\* \semiLamp{\bullet}}$
is homotopic to
%(by re-choosing a  basis of $\fing$ if necessary)
\begin{align*}
%(Q)_{(0)}^{M\* \semiLamp{\bullet}}=\affQ{-}
%=
\affQ{-}=\affQ{-}^{\st} +\chi_-,
\end{align*}
where 
\begin{align}
\affQ{-}^{\st}=&\sum_{n\in \Z\atop
\alpha\in \sroots_{<0}}J_{\alpha}(-n)\psi_{-\alpha}(n)\\
&-\frac{1}{2}\sum_{k,l\in \Z\atop
\alpha,\beta,\gamma\in \roots_{<0}}c_{\alpha,\beta}^{\gamma}
\psi_{-\alpha}(-k)\psi_{-\beta}(-l)\psi_{\gamma}(k+l),
\nno \\
\chi_-= &\sum_{\alpha\in \sroots_{<0}}\chi(J_{-\alpha})\psi_{\alpha}(0).
\end{align}
One has
\begin{align*}
\affQ{-} ( M\* \semiLamp{i} )
\subset  M\* \semiLamp{i-1}.
\end{align*}
It follows that by
Proposition \ref{Pro:2007-12-18-03}
the homology space
\begin{align*}
 \BRST_{\bullet}(M):=H_{\bullet}(M\* \semiLamp{\bullet}, \affQ{-})
\end{align*}
can be  considered as a 
Ramond twisted representation of $\Wg$. 

The $\sigma_R$-twisted representation $M\* \semiLamp{\bullet}$
of $\bCl$ is graded by the Hamiltonian $-D$,
which acts on it diagonally.
Obviously $D$ commutes with $\affQ{-}$,
and hence 
$\BRST_{\bullet}(M)$ is graded by the Hamiltonian $-D$. 
It follows that
we have obtained  the functor
\begin{align}
 V^k(\fing)\adMod\ra \Wg\adMod, \quad M\mapsto \BRST_{0}(M).
\label{eq:BRST-reduction-functor}
\end{align}
\begin{Rem}
{\rm One has
$ \BRST_{\bullet}(M)=H_{\frac{\infty}{2}+\bullet}
(\L\sg_{<0}, M\* \C_{\chi_-})
$,
where 
the right-hand-side is the Feigin's semi-infinite 
$L\sg_{<0}$-homology \cite{Feu84}
 with the coefficient in the $L\fing_{<0}$-module
$M\* \C_{\chi_-}$,
and
 $\chi_-$ is identified with the character of
$L\fing_{<0}$
such that  $\chi_-(J_{-\alpha}(n))=\delta_{n,0}\schi(J_{\alpha})$.
}\end{Rem}

\subsection{Finite $W$-algebras as $H$-twisted Zhu Algebras}
\label{subsection:Zhu-vs-finiteW}
Let $M\in \Wg\adMod_{\sigma_R}$
and suppose that 
$M_{\tp}$ is concentrated in one degree: $M_{\tp}=M_{d_0}$.
Then
\begin{align*}
 \BRST_{\bullet}(M)=\bigoplus_{d\in d_0+\Z_{\geq 0}}\BRST_{\bullet}(M)_d,
\end{align*}
and therefore, 
\begin{align}
\BRST_{\bullet}(M)_{\tp}=\BRST_{\bullet}(M)_{d_0},
\end{align}
provided that  $\BRST_{\bullet}(M)_{d_0}\ne 0$.

In this case $\BRST_{\bullet}(M)_{\tp}$ is easily described as follows:
Identify the Grassmann algebra $\bigwedge^{\bullet}(\sn_-)$
of $\finn_-$ with the subalgebra of $\bCl$ generated by $\psi_{\alpha}(0)$
with
$\alpha\in \sroots_-$.
Then  $\bigwedge^{\bullet}(\finn_-)$
is also identified with
the
 subspace $\semiLamp{\bullet}_{\tp}$
 of $\semiLamp{\bullet}$.
One has
\begin{align}
 \BRST_{0}(M)_{\tp}=H_0(M_{\tp}\* \bigwedge\nolimits^{\bullet}(\sn_-), 
\affQ{-}).
\end{align}
One sees that
the operator $\affQ{-}$
acts on $M_{\tp}\* \bigwedge^{\bullet}(\finn_-)$
as 
\begin{align}
 \sQ \label{eq:def-of-Q-fd-0}
=&\sum_{\alpha\in \sroots_{<0}}
(J_{\alpha}(0)-\chi(J_{-\alpha}))\psi_{-\alpha}(0)\\
&\qquad 
-\frac{1}{2}\sum_{\alpha,\beta,\gamma\in \sroots_{<0}}
c_{\alpha,\beta}^{\gamma}\psi_{-\alpha}(0)\psi_{-\beta}(0)\psi_{\gamma}(0).
\nno
\end{align}
From this formula it follows that the complex
$(M_{\tp}\* \bigwedge^{\bullet}(\finn_-), \affQ{-})$
is identical to  the
Chevalley-Eilenberg complex 
which defines the Lie algebra $\fing_{<0}$-homology 
$H_{\bullet}^{\Lie}(\fing_{<0}, M_{\tp}\* \C_{\schi_-})$
with the coefficient in the $\fing_{<0}$-module
$M\* \C_{\schi_-}$,
where 
$\C_{\schi_-}=U(\fing_{<0})/\ker \schi_-$
and
$\schi_-$ is the character of $\fing_{<0}$
defined by 
\begin{align*}
\chi_-(J_{\alpha})=\chi(J_{-\alpha}).
\end{align*}
Thus one has
\begin{align}
 \BRST_{\bullet}(M)_{\tp}=
H^{\Lie}_{\bullet}(\finn_-, M_{\tp}\* \C_{\chi_-}).
\label{eq:iso-top}
\end{align}
This in particular means that
$H^{\Lie}_{\bullet}(\finn_-, M_{\tp}\* \C_{\chi_-})$
is a module over
$\Zh_H(\Wg)$.

Recall \cite{De-Kac06}
that 
\begin{align}
\Zh_H(\Wg)\cong \Wfin,
\label{eq:Zhu=finW}
\end{align}
where $\Wfin$ 
is the {\em finite $W$-algebra}
associated with $(\fing,f)$.
The finite $W$-algebra $\Wfin$ may be defined by means of 
the quantum BRST reduction \cite{KosSte87,DAnDe-De-07}:
Let 
$\sCl$
be the Clifford algebra
associated with
$\sg_{<0}\+\sg_{>0}$
and $(~|~)_{|\sg_{<0}\+\sg_{>0}}$.
We identify 
$\sCl$ with the subalgebra
of $\bCl$ generated by $\psi_{\alpha}=\psi_{\alpha}(0)$
with $\sroots_{\ne 0}$.
One has the 
subalgebra $U(\fing)\* \sCl$
in $U_k(\affg)\* \bCl$,
and $\sQ$ can be considered 
as an odd element of $U(\fing)\* \sCl$.
One has $(\sQ)^2=0$,
and thus
\begin{align*}
(\ad \sQ)^2=0.
\end{align*}
Therefore 
$(U(\fing)\* \sCl,\ad \sQ)$ is a  chain
complex (with respect the grading by charge).
The corresponding
homology
\begin{align}
 H_{\bullet}(U(\fing)\* \sCl)
=H_{\bullet}(U(\fing)\* \sCl, \ad \sQ)
\end{align}
is  naturally a $\Z$-graded superalgebra.
\begin{Th}[{\cite{DAnDe-De-07}, cf.\ Theorem 2.4.2 of \cite{Ara07}}]
\label{Th:BRST-finite}
$ $

\begin{enumerate}
 \item It holds that
$H_{i}(U(\fing)\* \sCl)=0 
$  for all $i\ne 0$.
\item There is an algebra isomorphism 
$H_{0}(U(\fing)\* \sCl) 
\cong \Wfin$.
\end{enumerate}\end{Th}
For a $\fing$-module $M$,
 $M \* \Lam(\fing_{<0})$ is naturally
a $U(\fing)\* \sCl$-module.
Therefore
the algebra $H_0(U(\fing)\* \sCl)$
naturally acts on $H_{\bullet}^{\Lie}(\fing, M\* \C_{\schi_-})$.
As in the same manner as \cite{Ara07},
it follows that 
the action of $\Zh_H(\Wg)$ on $\BRST_{\bullet}(M)_{\tp}$
coincides with the action of 
$H_0(U(\fing)\* \sCl)$ on the space
$H_{\bullet}^{\Lie}(\fing, M_{\tp}\* \C_{\schi_-})$,
via the isomorphisms (\ref{eq:iso-top})
and (ii) of Theorem \ref{Th:BRST-finite}.

\section{Representation Theory of Affine $W$-Algebras
via the BRST Cohomology Functor}
\subsection{The Vanishing of the Lie Algebra Homology
%Preliminaries on  Representations of 
%Finite $W$-Algebras
}
Recall the notation from \S \ref{subsection:sl2-triple} and 
\S \ref{subsection:Root-Data}.

Let $\sL(\slam)$ be the irreducible highest weight
representation of $\fing$
with highest weight $\slam\in \dual{\sh}$.
 
Let
$\sBGG$ be the 
full subcategory of  
the category
of finitely generated left $\sg$-modules 
consisting of objects $M$
such that 
(1) $\dim U(\sn_+)m<\infty $ for all $m\in M$,
(2) $\sh$ acts semisimply on $M$,
(3) $M$ is a direct sum of 
 finite-dimensional
$\sg_0$-modules.

Set 
\begin{align}
\sP_0^+=\{\slam\in \dual{\sh};
\bra \slam,\alpha\che\ket\in \Z_{\geq 0}
\text{ for all $\alpha\in \sroots_{0,+}$}\}.
\end{align}
For $\slam\in \sP_0^+$
put
$\sM_0(\slam)=U(\sg)\*_{U(\sg_{\geq 0})}\sE(\slam)$,
where $\sE(\slam)$ is the irreducible 
finite-dimensional representation
of $\sg_0$ with  highest weight $\slam$,
considered as a $\sg_{\geq 0}$-module on which $\sg_{>0}$
acts trivially.
The 
$\sM_0(\slam)$ has
$\sL(\slam)$ as its unique simple quotient.
Every simple object of $\sBGG$
is isomorphic to exactly one of the
$\sL(\slam)$ with $\slam\in \sP_0^+$.

For a finitely generated 
$\sg$-module $M$
let $\Dim M$ be the Gelfand-Kirillov dimension of $M$.
%that is,
%the dimension of the associated
%variety $\V(M)$ of $M$
%(=the zero set of the vanishing ideal of
%$S(\sg)$-module $\gr M$).
%By definition,
%$\V(M)\subset \sg_{<0}^*$
%for $M\in \sBGG$.
%Thus
By (\ref{eq:d=half-of-dim-of-orbit}),
one has
\begin{align}
\Dim M\leq 
d_{\schi}
\end{align}
 for all $M\in \sBGG$.

Set
\begin{align}
\sH_{\bullet}(M)=H^{\Lie}_{\bullet}(\fing_{<0},
M\* \C_{\schi_-}).
\end{align} 
One sees that $\sH_0(M)$
is finite-dimensional
for any object $M$ of $\sBGG$ as in  Lemma 2.5.1 of \cite{Ara07}.
From \S \ref{subsection:Zhu-vs-finiteW}
it follows that  the correspondence
$M\mapsto \sH_0(M)$
defines a functor from $\sBGG$
to $\fdW$,
 the category of finite-dimensional
representations of $\Wfin$.

\begin{Th}[{Matumoto \cite{Mat90}}]\label{Th:exact-fd}
$ $

\begin{enumerate}
 \item  The functor $\sH_0(?): \sBGG\ra \fdW$
is exact.
\item 
Let
 $M$ be an object of $ \sBGG$.
One has
$\sH_0(M)\ne 0$
if and only if $\Dim M=d_{\chi}$.
\end{enumerate}\end{Th}
Because every projective object of $\sBGG$ is free over $U(\sg_{<0})$,
the following  assertion
follows from (i) of Theorem \ref{Th:exact-fd}
in the same manner as
Theorem 2.5.6 of 
 \cite{Ara07}.
\begin{Th}\label{Th:vanishing-finite}
One has
$\sH_i(M)=0$ for all $i\ne 0$
and for all $M\in \sBGG$.
\end{Th}

%\begin{Rem}
%{\rm
% In the case that $\schi$ is principal,
%it is known by Kostant
%\cite{Kos78} that
%$\Wg\cong \Center(\sg)$,
%the center of $U(\sg)$.
%(It is known that
%this is the only case when $\Wfin$ is commutative,
%see \cite{Pre07,PanPreYak06}.)
%Therefore 
%every simple $\Wg$-module is one-dimensional.
%In this case it is straightforward to see from (ii) of 
%Theorem \ref{Th:exact-fd}
%that non-zero $\sH_0(\sL(\slam))$
%is a simple $\Wg$-module
%and every simple $\Wg$-module appears
%in this way (cf. \S 2 of \cite{Ara07}).
%}\end{Rem}

\subsection{Representations of 
Finite $W$-Algebras in Type $A$}
\label{subsection:type-A-fd}
In \cite{BruKle05},
Brundan and Kleshchev  
gave a complete description of
irreducible finite-dimensional representations of
 $\Wfin$
in
type $A$,
as we recall below:

Let $\sg=\mathfrak{sl}_n(\C)$.
As usual,
we write
$\sroots=\{\alpha_{i,j}; 1\leq i,j\leq n\}$,
$\sroots_+=\{\alpha_{i,j}; 1\leq i<j\leq n\}$.

Let $Y_f$ be the partition $(p_1\leq p_2\leq \dots \leq p_r)$ of $n$
corresponding to the nilpotent element $f$.
Following  \cite{BruKle05},
we identify $Y_f$ with the Young diagram
with $p_i$ boxes in the $i$th row,
and
number the boxes of $Y_f$
by $1$,
$2$,
\dots , $n$ down columns from left to right.
%and let $\row(i)$ and $\col{i}$ denote the row and 
%column numbers if the $i$th box
The corresponding good grading is defined so that
\begin{align}
\sroots_{0}=\left\{\alpha_{i,j};
\text{the $i$th and the $j$th boxes belong to the same column} 
\right\}
\end{align}
(see \cite{ElaKac05,BruKle05} for details).
%Thus
%\begin{align}
%\sP_0^+=\{\slam\in \dual{\sh}; 
%\bra \slam,\alpha_{i,j}\ket\in \Z_{\geq 0}
%\ \text{if the $i$th and the $j$th boxes belong to the same column}
%\}.
%\end{align}
Let 
\begin{align}
 \sroots^f=\{\alpha\in \sroots;
\alpha(h)=0\ \forall h\in \sh^f\},
\end{align}
$\sroots^f_+=\sroots^f\cap \sroots_+$.
It is easy to see that
\begin{align*}
 \sroots^f=\{\alpha_{i,j}\in \sroots;
\text{the $i$th and the $j$th boxes belong to the same row}\}.
\end{align*}
Let
\begin{align}
 \sW^f=\{w\in \sW; w(h)=h\ \forall h\in \sh^f\}.
\end{align}
Then $\sW^f$
is the subgroup of
$\bar W=\mathfrak{S}_n$ 
generated by $s_{\alpha}$ with $\alpha\in \sroots^f$.

\begin{Th}[{Brundan and Kleshchev \cite{BruKle05},
$\sg=\mathfrak{sl}_n(\C)$}]
\label{Th:Brundan-Kleshchev}
%{\rm (type $A$)}
$ $

\begin{enumerate}
 \item For $\slam \in \sP_0^+$, 
$\sH_0(\sL(\slam))\ne 
 0$ if and only if
$\bra \slam+\bar\rho,\alpha\che\ket
\not\in \N$
for all $\alpha\in \sroots^f_+$.
In this case 
$\sH_0(\sL(\slam))$ is  irreducible.
Further, any irreducible finite-dimensional representation
of
$\Wfin$ arises in this way.
\item 
Nonzero %irreducible representations
$\sH_0(\sL(\slam))$ and $\sH_0(\sL(\bar {\mu}))$,
with  $\slam, \bar{\mu}\in \sP_0^+$,
are isomorphic
if and only if $\bar \mu+\bar\rho\in \sW^f(\slam+\bar \rho)$.
\end{enumerate}
\end{Th}
%\begin{Rem}
%{\rm It is not true  for 
%other types
% that
%nonzero $\sH_0(L(\lam))$ is
%irreducible over $\Wfin$,
%see Theorem 3.6.3 of \cite{Mat90}.
%}%However it is likely true that
%%$\sH_0(\slam)$ is a direct sum of irreducible representations
%of the same dimension.  
%\end{enumerate} }
%$\end{Rem}

\subsection{The Category $\affBGG$ of $\affg$-Modules}

Recall the notation from \S \ref{subsection:Kac-Moody}.

For $\lam\in \dual{\h}$
let $L(\lam)$ be the irreducible representation
of $\affg$ with  highest weight $\lam$.

Let
$\affBGG$ be the
full subcategory of the category of left $\affg$-modules
consisting of objects $M$
such that the following hold:
\begin{itemize}
 \item  $K$ acts as the multiplication by $k$ on $M$;
\item  $M$ admits a weight space decomposition;
\item
 there exists
a finite subset
$\{\mu_1,\dots,\mu_n\}$
of $
\dual{\h}_{k}$
such that
$M=\bigoplus\limits_{\mu\in \bigcup_i \mu_i-Q_+}M^{\mu}$;
\item for each $d\in \C$,
$M_d$ is an
 objects of $\sBGG$
as $\fing$-modules.
\end{itemize}
Set
\begin{align}
 \affP=\{\lam\in \dual{\bh}_k;\slam\in \sP_0^+,
\ \bra \lam,K\ket=k\}.
\end{align}
For $\lam\in \affP$,
let
\begin{align}
 M_0(\lam)=U(\affg)\*_{U(\sg[t]\+ \C K\+ \C D)
}\sM_0(\slam),
\end{align}
where
$\sM_0(\slam)$ is considered as a $\sg[t]\+ \C K \+ \C D$-module
on which $\sg[t]t$ acts trivially,
and $K$ and $D$ act the multiplication
by $\bra \lam,K\ket$ and $\bra \lam,D\ket$,
respectively.
The $M_0(\lam)$ is an object of $\affBGG$,
%and is
%called the {\em generalized Verma module}
%with highest weight $\lam$,
and has $L(\lam)$ as its
unique simple quotient.
Every irreducible object of 
$\BGG_k$ is isomorphic to exactly one of 
the $L(\lam)$ with $\lam\in \affP$.

The correspondence
$M\mapsto M^*$
defines a duality functor on $\affBGG$.
Here, $\affg$ acts on $M^*$
by
\begin{align}
(Xf)(v)=f(X^tv)
\label{eq:2006-10-16-00-24}
\end{align}
where
$X\mapsto X^t$ is  the 
anti-automorphism of $\affg$ define
by 
$K^t=K$, $D^t=D$ and $J(n)^t=J^t(-n)$ for $J\in \sg$, $n\in \Z$.

Clearly, $(M^*)^*=M$ for $M\in \affBGG$.
It follows that
 $L(\lam)^*=L(\lam)$.

Let
$\affBGG^{\tri}$
be the full subcategory of $\affBGG$
consisting of objects $M$ that admit 
a finite filtration 
$M=M_0\supset M_1\supset \dots \supset M_r=0 $
such that each successive subquotient $M_i/M_{i+1}$
is isomorphic to
some generalized Verma module $M_0(\lam_i)$ with $\lam_i\in \affP$.
Dually,
let $\affBGG^{\bigtriangledown}$  be 
the full subcategory of $\BGG_k$
consisting of objects $M$ such that
$\dual{M}\in \mathbb{O}bj\BGG^{\tri}_k$.

For $\lam\in \affP$,
let
$\affBGG^{\leq \lam}$
be the Serre full subcategory of $\affBGG$
consisting of objects $M$ 
such that
$M=\bigoplus\limits_{\mu\leq \lam}M^{\mu}$.
It is well-known (see e.g., \cite{Soe98})
that every $L(\mu)$
that lies in $\affBGG^{\leq \lam}$
admits the indecomposable projective
cover $P_{\leq \lam}(\mu)$
in $\affBGG^{\leq \lam}$,
and hence,
every finitely generated object in $\affBGG^{\leq \lam}$
is an image of a projective object
of the form
$\bigoplus_{i=1}^{r}
P_{\leq \lam}(\mu_i)$.
The
$P_{\leq \lam}(\mu)$
is an object of 
$
\affBGG^{\tri}$.
Dually,
$I_{\leq \lam}(\mu)=P_{\leq \lam}(\mu)^*$
is
the injective envelope of $L(\mu)$
in $\affBGG^{\leq \lam}$.
%In particular
% $M\in \Obj\affBGG^{\leq \lam}$
%is a submodule of an injective object of the form
%$\bigoplus_{i=1}^r
%I_{\leq \lam}(\mu_i)$
%if its dual $M^*$ is
%finitely generated.

\subsection{The ``Top'' Part of the 
BRST Cohomology}
\label{subsection:top-part-of-BRST}
Let $M$ be an object of $\affBGG$.
Clearly, $M_{\tp}$ is a $\sg$-submodule of $M$.
By Theorem \ref{Th:vanishing-finite},
$\sH_i(M_{\tp})=0$ for all $i>0$,
and $\sH_0(M_{\tp})$ is a
finite-dimensional  
$\Wfin$-module.

The following assertion follows 
from 
Theorem \ref{Th:vanishing-finite}.
\begin{Lem}\label{lem:top-identical}
 Let $M$ be an object of $\affBGG$.
Assume that $\sH_{\bullet}(M_{\tp})\ne 0$.
Then 
one has
the following
 isomorphism of  $\Wfin$-modules:
\begin{align*}
\BRST_{i}(M)_{\tp}\cong \begin{cases}
		      \sH_{0}(M_{\tp})&
\text{for $i=0$}\\
0&
\text{for $i\ne 0$}.
		     \end{cases}
\end{align*}
\end{Lem}
The following assertion follows  from 
Theorems \ref{Th:exact-fd},
\ref{Th:vanishing-finite} and 
Lemma \ref{lem:top-identical}.
\begin{Pro}
 \label{Pro:BRST-top}
One has
\begin{align*}
&\BRST_{i}(M_0(\lam))_{\tp}\cong \begin{cases}
		      \sH_{0}(\sM_0(\slam))&
\text{for $i=0$}\\
0&
\text{for $i\ne 0$},
		     \end{cases}\\
&\BRST_{i}(M_0(\lam)^*)_{\tp}\cong \begin{cases}
		      \sH_{0}(\sM_0(\slam)^*)&
\text{for $i=0$}\\
0&
\text{for $i\ne 0$},
		     \end{cases}
\end{align*}
and if  $\Dim  \sL(\slam)=d_{\schi}$, then
\begin{align*}
\BRST_{i}(L(\lam))_{\tp}\cong \begin{cases}
		      \sH_{0}(\sL({\slam}))&
\text{for $i=0$}\\
0&
\text{for $i\ne 0$.}
		     \end{cases}
\end{align*}
\end{Pro}

\subsection{The Vanishing and the Almost Irreduciblity}
\begin{Th}\label{Th:estimate}
Let $M$ be an object of $\affBGG$.
Then  $\BRST_{\bullet}(M)_d$ is finite-dimensional for all $d$.
If $M$ is an object of $\affBGG^{\leq \lam}$
then
$\BRST_i(M)_d=0$ unless $|i|\leq d-\bra \lam,D\ket$.
\end{Th}
\begin{proof}
By  Theorem \ref{Th:vanishing-finite}
one has
 $\sH_{i}(M|_{\fing})=0$
for all $i\ne 0$.
Therefore by considering the Hochschild-Serre spectral
sequence for %the subalgebra 
$\sg_{<0}\subset L{\sg_{<0}}$,
the assertion follows
in the same manner as 
Theorem 7.4.2 of
\cite{Ara07}.
\end{proof}
Theorem \ref{Th:estimate}
in particular implies that
 $\BRST_{\bullet}(M)$ is an ordinary representations
for all $M\in \affBGG$. 
It follows that
one has the functor
\begin{align}
 \affBGG\ra \Wg\Mod_{\sigma_R},\quad M\ra \BRST_0(M).
\label{eq:--reduction-functor-ordinary}
\end{align}

\begin{Th}[\cite{KacWak04}]\label{Th:vanishing-Verma}
For $\lam\in \affPp$ one has the following:
\begin{enumerate}
 \item 
$\BRST_i(M_0(\lam))=0$ for all $i\ne 0$.
\item 
 $\BRST_0(M_0(\lam))$ is 
almost highest weight.
%generated by
%the subspace
%$\BRST_0(M_0(\lam))_{\tp}$.
\end{enumerate}
\end{Th}
(The proof of Theorem \ref{Th:vanishing-Verma} %the following assertion 
 is the same as that of 
Theorem \ref{Th:KW}.)
\begin{Th}\label{Th:Vanishing-verma-dual}
For $\lam\in \affPp$ one has the following:

 \begin{enumerate}
\item  $\BRST_i(M_0(\lam)^*)=0$ for all $i\ne 0$.
\item 
 $\BRST_0(M_0(\lam)^*)$ is almost co-highest weight.
 \end{enumerate}
\end{Th}
The proof of 
Theorem \ref{Th:Vanishing-verma-dual}
is given in
Section  \ref{section:proof}.

Though our formulation is slightly different from that of \cite{KacWak},
the following assertion
essentially confirms Conjecture B of \cite{KacWak},
partially (cf.\ Theorems \ref{Th:Main:typeA}, 
 \ref{Th;irr-generic} and \ref{Th:moruldar-invariant-represenrations}
below).
\begin{Th}[The main result]\label{Th:Main}
 Let $k$ be any complex number.
\begin{enumerate}
 \item Let $M$ be an object of $\affBGG$.
Then $\BRST_i(M)=0$ for all $i\ne 0$.
In particular
the functor $\BRST_0(?): \affBGG\ra \Wg\Mod_{\sigma_R}$ is exact.
\item For $\lam\in \affPp$,
 $\BRST_0(L(\lam))$ is zero or almost irreducible.
Further, one has
 $\BRST_0(L(\lam))\ne 0$ if and only if
$\Dim \sL(\slam)=d_{\chi}$.
\end{enumerate}
\end{Th}
\begin{proof}
We give only the sketch
of the proof because it is essentially
the  same as 
those of   
Theorems 
7.6.1 and 7.6.3 of \cite{Ara07}.

From Theorem \ref{Th:vanishing-Verma}
(i)
it follows that
$\BRST_i(M)=0$ for all $i\ne 0$
and $M\in \affBGG^{\Delta}$,
and hence $\BRST_i(P_{\leq \lam}(\mu))=0$
for all $i\ne 0$
and all $\mu\leq \lam$ in $\affPp$.
This together with Theorem \ref{Th:estimate}
gives the vanishing of $\BRST_i(M)$
for all $i>0$ and all $M\in \affBGG$.
Likewise, Theorem \ref{Th:Vanishing-verma-dual} (i)
gives $\BRST_i(M)=0$ for all $i<0$ and all
$M\in \affBGG$. This shows (i).
(ii) follows from (i) using 
Theorem \ref{Th:exact-fd} (ii),
Theorem \ref{Th:vanishing-Verma} (ii)
and Theorem \ref{Th:Vanishing-verma-dual} (ii).
\end{proof}

\begin{Co}\label{Co:criterion-irr}
Let $\lam\in \affPp$
with $k\in \C$.
The representation  
 $\BRST_0(L(\lam))$  is irreducible 
over $\Wg$
if and only if 
$\sH_0(\sL(\slam))$ is irreducible over $\Wfin$.
\end{Co}
\begin{proof}
 The assertion follows immediately from Proposition \ref{Pro:BRST-top}
and Theorem \ref{Th:Main} (ii).
\end{proof}
%\begin{Rem}
%{\rm  Though our formulation is slightly different from that of \cite{KacWak},
%Theorem \ref{Th:Main}
%essentially confirms Conjecture B of \cite{KacWak},
%partially (cf.\ Theorems \ref{Th:Main:typeA}, 
% \ref{Th;irr-generic} and \ref{Th:moruldar-invariant-represenrations}
%below).
%}\end{Rem}

\subsection{The Character of $\BRST_0(L(\lam))$}
Let $\ch L(\lam)$ be the character of $L(\lam)$:
$\ch L(\lam)=\sum_{\mu}e^{\mu}\dim L(\lam)^{\mu}$.
One has
\begin{align*}
\ch L(\lam)=\sum_{\mu\in \dual{\bh}}c_{\lam, \mu}
\frac{e^{\mu}}
{\prod_{\alpha\in \roots_+}(1-e^{-\alpha})^{\-\dim \affg_{\alpha}}}
\end{align*}
with some $c_{\lam,\mu}\in \Z$.
The coefficient $c_{\lam,\mu}$ is known by Kashiwara
and Tanisaki
\cite{KasTan00} (in terms of the Kazhdan-Lusztig polynomials)
provided that $k$ is not critical 
(for any simple summand of $\fing$).

Recall \cite{KacWak04,KacWak}
that the ``Cartan subalgebra'' of $\Wg$
is given by
\begin{align}
 \afft=\bar\afft\+ \C D,
\quad\text{where }\bar \afft=\sh^f.
\end{align}
Because  it commutes with $\affQ{-}$,
 $\afft$ acts on the space $\BRST_{\bullet}(M)$.

Let  \begin{align*}
\ch \BRST_{\bullet}(L(\lam))=\sum_{\xi\in \dual{\afft}}e^{\xi}
\dim \BRST_{\bullet}(L(\lam))_{\xi},
     \end{align*}
where
$\BRST_{\bullet}(L(\lam))_{\xi}=\{ c\in \BRST_{\bullet}(L(\lam));
t c=\xi(t)c\ \forall t\in \afft\}$.

Set
\begin{align}
\chi_{\BRST_{\bullet}(L(\lam))}=
\sum_{i=-\infty}^{\infty}(-1)^i \ch \BRST_i(L(\lam)).
\end{align}By the Euler-Poincar\'{e} principle
one has \cite{FKW92,KacRoaWak03,KacWak}  
\begin{align}
 \chi_{\BRST_{\bullet}(L(\lam))}=
\frac{\sum_{\mu}c_{\lam,\mu} e^{\mu|{_\afft}}}
{\prod_{j\geq 1}(1-e^{-j\delta|_{\afft}})^{\rank \sg}
 \prod_{\alpha\in \rroots_{0,+}}(1-e^{-\alpha|_{\afft}})},
\label{eq:Euler-Poincare-ch}
\end{align}
where  $\rroots_{0,+}=\{\alpha\in \rroots_+;
\bar \alpha\in \sroots_0\}$.

The following assertion follows immediately from
Theorem \ref{Th:Main}.
\begin{Th}\label{Th:ch-formula}
For $\lam\in \affPp$
one has
\begin{align*}
 \ch \BRST_0(L(\lam))=\chi_{\BRST_{\bullet}(L(\lam))}.
\end{align*}
\end{Th}

\subsection{Type $A$ Case}
In type $A$,
the following assertion
follows immediately from
(\ref{eq:Zhu=finW}),
 Theorems \ref{Th:BRST-finite}, \ref{Th:Brundan-Kleshchev}
and  \ref{Th:Main} (in the notation of \S \ref{subsection:type-A-fd}).
\begin{Th}[$\fing=\mathfrak{sl}_n$]\label{Th:Main:typeA}
 Let $k$ be any complex number.
\begin{enumerate}
\item 
One has
$\BRST_i(M)=0$ for all $i\ne 0$ and  all $M\in \affBGG$.
 \item For $\lam \in \affPp$, 
$\BRST_0(\L(\lam))\ne 
 0$ if and only if
$\bra \slam+\bar\rho,\alpha\che\ket
\not\in \N$
for all $\alpha\in \sroots^f_+$.
In this case 
$\BRST_0(\L(\lam))$ is  irreducible.
Further, any irreducible ordinary Ramond twisted  representation
of
$\W^k(\mathfrak{sl}_n,f)$ arises in this way.
\item 
Nonzero %irreducible representations
$\BRST_0(\L(\lam))$ and $\BRST_0(\L( {\mu}))$
with  $\lam, {\mu}\in \affPp$
are isomorphic
if and only of $\bar \mu+\bar\rho\in \sW^f(\slam+\bar \rho)$.
\end{enumerate}
\end{Th}

Theorems \ref{Th:ch-formula} and \ref{Th:Main:typeA} determine\footnote{
In the case of $f$ is a principal nilpotent element
the characters of all irreducible positive energy representations
of $\Wg$ was previously determined in \cite{Ara07}
(for all $\fing$ and all $k\in \C$).
Also,
in the case $f$ is a minimal nilpotent element
the characters of all 
irreducible 
(non-twisted) 
positive energy representations
of $\Wg$ was previously determined in \cite{Ara05}
(for all $\fing$ and all non-critical $k$).}
 the 
characters of all
 irreducible ordinary 
Ramond twisted representations of 
$\W^k(\mathfrak{sl}_n,f)$ 
for all nilpotent elements $f$
at all non-critical levels $k$.

\begin{Rem}
{\rm  If $\fing$ is not of type $A$,
it not true that
nonzero $\BRST_0(L(\lam))$ is always irreducible,
see
 Theorem 3.6.3 of \cite{Mat90}.
However it is likely that
$\BRST_0(L(\lam))$ is a direct
sum of irreducible modules.       
} 
\end{Rem}

\subsection{Modular Invariant Representations}
\label{eq:Modular-Invariant-Represenrations}
In this section we assume that $\fing$ is simple.

Let $\Pr^k$ be the set of principal 
admissible weights \cite{KacWak89,KacWak}
of $\affg$
of level $k$.
For $\lam\in \Pr^k$
one has \cite{KacWak88}
\begin{align}
 \ch L(\lam)=\sum_{w\in W(\lam)}(-1)^{\ell_{\lam}(w)}
\frac{e^{w\circ \lam}}{\prod_{j\geq 1}(1-e^{-j\delta})^{\rank \fing} 
\prod_{\alpha\in \rroots_{+}}(1-e^{-\alpha}).}
\label{eq:ch-formala-admissible}
\end{align}

Let $\sroots(\lam)=\roots(\lam)\cap \sroots$,
and]let
 $\sW(\slam)\subset \sW$
be  the integral Weyl group of $\slam\in \dual{\sh}$
generated by $s_{\alpha}$ with $\alpha\in \sroots(\lam)$.
The formula (\ref{eq:ch-formala-admissible})
in particular
implies that
\begin{align}
 \ch \sL(\slam)=\sum_{w\in \sW(\slam)}(-1)^{\ell_{\slam}(\slam)}
\frac{e^{w\circ \lam}}{\prod_{\alpha\in \sroots_+}
(1-e^{-\alpha}).}
\end{align}

We remark that an element $\lam$ of $\Pr^k$
done not necessarily 
belong to
$\affPp$. 
However 
the Euler-Poincar\'{e} character $\chi_{\BRST_0(L(\lam))}$
makes sense for all $\lam\in \Pr^k$ \cite{KacWak},
and coincides with  the right-hand-side of (\ref{eq:Euler-Poincare-ch}).
Thus
it
has the form
\begin{align}
 \chi_{\BRST_0(L(\lam))}=
e^{\bra \lam,D\ket\delta|_{\afft} }\sum_{j\in \Z_{\geq 0}}
e^{-j\delta|_{\afft} }\varphi_{\lam,j}
\end{align}
with 
\begin{align}
 \varphi_{\lam,0}=
\frac{\sum_{w\in \sW(\slam)}(-1)^{\ell_{\lam}(w)}e^{w\circ \lam|_{\afft}}}
{\prod_{\alpha\in \sroots_{0,+}}(1-e^{-\alpha|_{\afft}}).}
\end{align}
Note that 
$\varphi_{\lam,0}$
is the Euler-Poincar\'{e} character of
$\sH_{\bullet}(\sL(\slam))$.

The Euler-Poincar\'{e} character $\chi_{\BRST_{\bullet}(L(\lam))}$
is called {\em almost convergent} \cite{KacWak} 
if $\lim\limits_{z\ra 0}\varphi_{\lam,0}(z)$
$(z\in  \bar \afft)$
exists and is non-zero.
Set
\begin{align}
& \widetilde{M}_k=\{\lam\in \Pr^k;
\chi_{\BRST_{\bullet}(L(\lam))}
\text{ is almost convergent}\},\\
&M_k=\widetilde{M}_k\cap \affPp.
\end{align}
(We are not assuming that $\fing=\mathfrak{sl}_n$.)

\begin{Th}\label{Th;irr-generic}
Let $\lam\in  {M}_k$.
%Suppose that $\chi_{\BRST_{\bullet}(L(\lam))}$ 
%is almost convergent.
Then 
$\BRST_{\bullet}(L(\lam))$ is irreducible.
\end{Th}
\begin{proof}
 By Corollary   \ref{Co:criterion-irr}
it is sufficient to show that
$\sH_0(\sL(\slam))$ is irreducible over $\Wfin$.

%Set $\sroots(\lam)=\{\alpha\in \sroots;
%\bra \lam+\rho,\alpha\che\ket\in \Z\}$.
By Corollary 2.2 of \cite{KacWak} (or its proof)
one has
\begin{align*}
|\sroots(\lam)|=|\sroots_0|.
\end{align*}
(In our setting $\roots^0\sqcup \roots^{1/2}$
in \cite{KacWak} is identified with $\sroots_0$,
see \cite{BruGoo07}.)
Because $\lam\in \affPp$,
$\sroots_0\subset \sroots_+(\lam)$,
and hence 
$\sroots(\lam)=\sroots_0$.
This implies \begin{align}
	\bra \lam+\rho,\alpha\che\ket\not\in \Z,
\quad \forall \alpha\in \sroots_{>0}.
\label{eq:generic}
	     \end{align}
Thanks to Theorem 3.4.4 of \cite{Mat90},
this gives the irreduciblity of   $\sH_0(\sL(\slam))$.
\end{proof}

\begin{Rem}
 {\rm Let $\lam\in \affPp$.
From Theorem \ref{Th:Main}
it follows that
$\chi_{\BRST_0(L(\lam))}$
is almost convergent if and only if
$\Dim \sL(\slam)=d_{\chi}$.
}
\end{Rem}

\smallskip

Recall  \cite{KacWak} 
that the pair $(k,f)$  is called
{\em exceptional} 
if 
the Euler-Poincare character 
$\chi_{\BRST_{\bullet}(L(\lam))}$ is almost convergent
for some $\lam\in \Pr^k$,
and is either zero or
almost convergent for all $\lam\in \Pr^k$.

The exceptional pairs are classified in \cite{KacWak}
in type $A$:
Each admissible number \cite{KacWak89} $k$
of $\mathfrak{sl}_n$
is written as
\begin{align}
k+n=\frac{p}{q}
,\quad
p\geq n,\quad 
q\geq 1,\quad
(p,q)=1.
\label{eq:admissible-number}
\end{align}For such a $k$
the pair
$(k,f)$ is exceptional if and only if
$f$ is the  nilpotent element corresponding
 to the partition $(s,q,q,\dots, q)$
($s\equiv n\pmod{q}$,
$0\leq s<q$).

The following assertion was 
implicitly proved\footnote{In the case that
$f$ is a principal nilpotent element 
($=$ the case that $q\geq n$, 
$\sroots_0=\emptyset$ and $\sroots^f=\sroots$)
Proposition \ref{Pro:Wf-action}
was proved in \cite{FKW92}.
} in \cite{KacWak}.
\begin{Pro}\label{Pro:Wf-action}
 Let $(k,f)$ be an exceptional pair for $\mathfrak{sl}_n$.
There is an bijection
\begin{align*}
 \sW^f\times M_k\isomap \widetilde{M}_k,
\quad (w,\lam)\mapsto w\circ \lam.
\end{align*}
\end{Pro}
\begin{proof}
%First, we show that
%$\sW^f$ acts faithfully on $\widetilde{M}_k$.
By Theorem 2.3 of \cite{KacWak},
\begin{align}
 \widetilde{M}_k=\{\lam\in \Pr^k;
\sroots(\lam)\subset \sroots\backslash \sroots^f\}.
\label{eq:Theorem 2.3 of Kac Wakimoto}
\end{align}
Let $\lam\in \widetilde{M}_k$,
$w\in \sW^f$.
Since $\rroots_+\cap w\inv (\rroots_-)\subset \sroots^f_+$,
\eqref{eq:Theorem 2.3 of Kac Wakimoto}
gives 
$\rroots_+(\lam)\cap w\inv(\rroots_-)=\emptyset$,
or equivalently,
 $w\circ \lam\in \Pr^k$.
Because
\begin{align}
  \chi_{\BRST_{\bullet}(L(\lam))}=\chi_{\BRST_{\bullet}(L(w\circ
  \lam))},
\quad 
 \forall w\in \sW^f,
\end{align}
the element  $w\circ \lam$ belongs to $\widetilde{M}_k$.
Therefore
the shifted action of
$\sW^f$ preserves
 $\widetilde{M}_k$.
Further, 
again by \eqref{eq:Theorem 2.3 of Kac Wakimoto},
it follows that
this action of $\sW^f$ on $\widetilde{M}_k$ is faithful,
and that
$M_k\cap( \sW^f \circ \lam)=\{\lam\}$
for $\lam\in M_k$.

Next 
let $k$ be as in (\ref{eq:admissible-number}).
By Lemma 3.1 of \cite{KacWak}
one has
\begin{align}
 \rank \sroots(\lam)\geq \text{min}(n-q,0)
=\rank \sroots_0,\quad \forall \lam\in \Pr^k.
\end{align}
According to (the proofs of)
Propositions 3.2 and 3.3 of \cite{KacWak},
the rank of any root subsystem in 
$\sroots\backslash \sroots^f$
is equal to or smaller than $\rank \sroots_0$, 
and is equal to $\rank \sroots_0$ 
if and only if it
is $\sW^f$-conjugate to 
$\sroots_0$.
Thus for $\lam\in \widetilde{M}_k$
there exists 
$w\in\sW^f$
such that
 $\sroots(\lam)=w(\sroots_{0})$,
and thus
$w\inv\circ \lam\in M_k$.
This completes the proof.
\end{proof}

According to \cite{KacWak},
 Theorem \ref{Th:Main:typeA}
and Proposition \ref{Pro:Wf-action}
give
 the following assertion\footnote{However the rationality of
the simple quotient
of 
 $\W^k(\mathfrak{sl}_n,f)$
still
remains  to be an open problem.}.
\begin{Th}[Conjectured by Kac and Wakimoto \cite{KacWak}]
\label{Th:moruldar-invariant-represenrations}
Let $(k,f)$ be an exceptional pair for $\mathfrak{sl}_n$.
The linear span of  the normalized characters of
irreducible ordinary Ramond twisted representations
$\BRST_0(L(\lam))$
of $\W^k(\mathfrak{sl}_n,f)$,
with $\lam\in M_k$,
are 
%\begin{align*}
% \{ \BRST_0(L(\lam)); \lam\in M_k\}
%\end{align*}
%are 
closed under the natural action of $SL_2(\Z)$.
\end{Th}
%\begin{Rem}
%{\rm 
% In the case that $f$ is a principal
%nilpotent element
%the existence of modular
%invariant representations 
%of $\Wg$ was   
%conjectured by Frenkel, Kac and Wakimoto \cite{FKW92}
%and proved in  \cite{Ara07},
%cf.\ \S 2.4 of  \cite{KacWak}.
%}\end{Rem}

\section{Proof of Theorem \ref{Th:Vanishing-verma-dual}}
\label{section:proof}
The proof of Theorem \ref{Th:Vanishing-verma-dual}
is essentially the repetition
of the argument  of \S 7 of \cite{Ara07}.
Therefore  we give only the sketch  of the proof. 
\subsection{Step 1}
Let  
\begin{align}
C^{\bullet}(\gVerma{\lam}):=\gVerma{\lam}\* \semiLamp{\bullet}.
\end{align}
As in \S 8.2 of \cite{Ara07},
we identify
%the space
$\gVerma{\lam}^*\* \semiLamp{\bullet}$
with %the graded  dual space 
$C^{\bullet}(\gVerma{\lam})^*$
($^*$ is defined in (\ref{eq:graded-dual})):
\begin{align}
 \BRST_{\bullet}(\gVerma{\lam}^*)=
H_{\bullet}(C^{\bullet}(\gVerma{\lam})^*,\affQ{-}).
\end{align}

The differential $\affQ{-}$ acts on
$C^{\bullet}(\gVerma{\lam})^*$
by 
\begin{align}
(\affQ{-}  \phi)(c)=\phi(\affQ{+} c)
\end{align}
for $\phi \in C^{\bullet}(\gVerma{\lam})^*$,
$v\in C^{\bullet}(\gVerma{\lam})$,
where
\begin{align}
\affQ{+}=(Q_+^{\st})_{(0)}+\chi_+',
\quad 
\chi_+'=\sum_{\alpha\in
\roots_{\geq 1}}\chi(x_{\alpha})\psi_{-\alpha}(0).
\end{align}

Below we regard
$C^{\bullet}(\gVerma{\lam})$  as a $\sigma_R$-twisted 
representation of $\bCg$
by
the 
action
\begin{align}
X(n)^R \mapsto   \widehat{t}_{-\frac{1}{2}h_0}(X(n)),
\quad
 \psi_{\alpha}(n)^R\mapsto\widehat{t}_{-\frac{1}{2}h_0}(\psi_{\alpha}(n))
\label{eq:2008-08-01-22-13-35}
\end{align}
(see Remark \ref{Rem:action-is-not-the-same} below).
Then 
\begin{align}
 (Q_{(-1)}\1)^{C^{\bullet}(\gVerma{\lam})}_{(0)}=\affQ{+}
\end{align}
(in the notation of \S  \ref{subsection:twisted-representations}).

Let $C_+^{\bullet}(\lam)$
the $\bCg_+$-submodule
of $C^{\bullet}(\gVerma{\lam})$
spanned by the vectors
\begin{align}
 \wJ_{a_1}(m_1)\dots \wJ_{a_r}(m_r)\psi_{\beta_1}(n_1)
\dots \psi_{\beta_s}(n_s)v_{\lam}
\end{align}
with $a_i\in \sroots_{\leq 0}\sqcup \sI$
and $\beta_i\in \sroots_{<0}$,
where $v_{\lam}$ is the highest weight vector
of $C^{\bullet}(\gVerma{\lam})$.
As in \S \ref{subsection:generating-field-of-W},
it follows that 
$C_+^{\bullet}(\lam)$ is an subcomplex 
of 
$C^{\bullet}(\gVerma{\lam})$.

The graded dual space $C_+^{\bullet}(\lam)^*
%=\bigoplus_{\mu}\Hom_{\C}(C_+^{\bullet}(\lam))^{\mu},\C)
$
of $C_+^{\bullet}(\lam)$ is a quotient complex of 
$C^{\bullet}(\gVerma{\lam})^*$.
%$\gVerma{\lam}^*\* \semiLamp{\bullet}$.
Thus   there is a natural  map
\begin{align}
\BRST_{\bullet}(\gVerma{\lam}^*)\ra
H_{\bullet}(C_+^{\bullet}(\lam)^*).
\label{eq:2008-01-13-1}
\end{align}
The
space
$C_+^{\bullet}(\lam)^*$ is a
$\mathcal{C}_+^{\bullet}$-submodule of 
$\gVerma{\lam}^*\* \semiLamp{\bullet}$
with respect to the action 
(\ref{eq:twisted-action}).
Hence by (\ref{eq:W-as-sub})
it follows that (\ref{eq:2008-01-13-1})
is a homomorphism of
Ramond twisted  representations of $\Wg$.

One has the following assertion
(cf.\ Proposition 8.3.4
of \cite{Ara07}):
\begin{Pro}
The map (\ref{eq:2008-01-13-1}) gives the  isomorphism
\begin{align*}
\BRST_{\bullet}(\gVerma{\lam}^*)
\cong H_{\bullet}(C^{\bullet}_+(\lam)^*)
\end{align*}
 of $\Wg$-modules.
\end{Pro}
\begin{Rem}
\label{Rem:action-is-not-the-same}
{\rm 
 Using the action (\ref{eq:2008-08-01-22-13-35})
one can define a $\sigma_R$-twisted $\bCg$-module structure on 
$M_0(\lam)^* \*\semiLamp{\bullet}$ by
the formula
\begin{align*}
(X(n) \phi)(c)=\phi(X^t(-n) c)  \quad \text{with $X\in \fing$.}
\end{align*}
This is not the same as action as \eqref{eq:twisted-action},
but as easily seen 
 $\BRST_{0}(\gVerma{\lam}^*)$  is almost co-highest weight
if and only if it is so with respect to this new
action of $\Wg$.
}\end{Rem}
\subsection{Step 2}
One has
\begin{align*}
C_+^{\bullet}(\lam)=\bigoplus\limits_{d\in
-\bra \lam,D\ket +\Z_{\geq 0}}C_+^{\bullet}(\lam)_d,
\quad \dim C_+^{\bullet}(\lam)_d=\infty.
\end{align*}
Note that
the subspace
 $C_+^{\bullet}(\lam)_{\tp}=C_+^{\bullet}(\lam)_{-\bra \lam,D\ket}$
 is 
the subcomplex of 
$(C_+^{\bullet}(\lam), \affQ{+})
$ spanned by the
 vectors
\begin{align}
 \wJ_{a_1}(0)\dots \wJ_{a_r}(0)
\psi_{\beta_1}(0)\dots \psi_{\beta_s}(0)v_{\lam}
\end{align}
with $a_i\in \sroots_{\leq 0}\sqcup \sI$,
$\beta_i\in \sroots_{<0}$,
and hence,
\begin{align}
C_+^{\bullet}(\lam)_{\tp}=\sM_0(\slam)\* \bigwedge\nolimits^{\bullet}
(\sg_{>0}^*).
\end{align}
One has the weight space
decomposition 
\begin{align*}
C_+^{\bullet}(\lam)_{\tp}=
\bigoplus\limits_{\mu\in \dual{\bh}
\atop \bra \lam-\mu,x_0\ket\geq 0}
C_+^{\bullet}(\lam)_{\tp}^{\mu}.
\end{align*}Define a 
decreasing filtration
 \begin{align*}
  C_+^{\bullet}(\lam)_{\tp}= F^0C_+^{\bullet}(\lam)_{\tp}
\supset F^1 C_+^{\bullet}(\lam)_{\tp} \supset \dots  
 \end{align*}
of $C_+^{\bullet}(\lam)_{\tp}$
by
\begin{align}
 F^p C_+^{\bullet}(\lam)_{\tp}=
\bigoplus\limits_{\mu\in \dual{\bh}
\atop  \bra \lam-\mu,x_0\ket\geq p}
C_+^{\bullet}(\lam)_{\tp}^{\mu}.
\end{align}
Then
\begin{align}
 &(Q_+^{\st})_{(0)}\cdot F^p C_+^{\bullet}(\lam)_{\tp}
\subset F^p C_+^{\bullet}(\lam)_{\tp},\\
&
\chi_+' \cdot F^p C_+^{\bullet}(\lam)_{\tp}\subset F^{p+1}
 C_+^{\bullet}(\lam)_{\tp}.
\label{eq:2008-01-22-13-55}
\end{align}

Let
$F^p C_+^{\bullet}(\lam)$ be the subspace
of $C_+^{\bullet}(\lam)$
generated by
$C_+^{\bullet}(\lam)_{\tp}$
over $\bCg_+$.
One has
 \begin{align}
&  C_+^{\bullet}(\lam)= F^0C_+^{\bullet}(\lam)
\supset F^1 C_+^{\bullet}(\lam)\supset \dots ,
\\& \bigcap_p F^p C_+^{\bullet}(\lam)=0,
\\
& \affQ{+} F^p C_+^{\bullet}(\lam)\subset F^p C_+^{\bullet}(\lam),\\
& a_{(n)}\cdot F^p C_+^{\bullet}(\lam)
\subset F^p C_+^{\bullet}(\lam)
\quad (a\in \bCg_+, n\in \Z)
\label{eq:2008-01-16-15-36}
 \end{align}
(cf. Proposition 8.5.3
of  \cite{Ara07}).

Let $(\cE_r^{p,q}, d_r)$ be the corresponding
spectral sequence:
\begin{align}
&\cE_0^{p,q}=F^p C^{p+q}_+(\lam)/F^{p+1}C^{p+q}_+(\lam),\\
&\cE_1^{p,q}=H^{p+q}(\cE^{p,\bullet}_0).
\end{align}
We do not claim that 
this spectral sequence 
 converges
to $H^{\bullet}(C_+^{\bullet}(\lam))$.
We will show in  Proposition \ref{Pro:cE-to-E} below
 that
$\cE_r$ converges to the dual 
$\sfD(\BRST_{\bullet}(\gVerma{\lam}^*))$
of $\BRST_{\bullet}(\gVerma{\lam}^*)$.

\subsection{Step 3}
Set
\begin{align}
 F_p C^{\bullet}_+(\lam)^*
=(C_+^{\bullet}(\lam)/F^p C_+^{\bullet}(\lam))^*
\subset C^{\bullet}_+(\lam)^*.
\end{align}
Then 
$\{F_p C^{\bullet}_+(\lam)^*\}$
defines an
exhaustive,
  increasing filtration of
the chain complex 
$\{C^{\bullet}_+(\lam)^*\}$ 
which is obviously bounded below
(cf.\ 
Lemma 8.5.4 and Proposition 8.5.5
of \cite{Ara07}).
It follows that
one has the corresponding 
converging
spectral sequence
\begin{align}
 E^r\Rightarrow H_{\bullet}(C^{\bullet}_+(\lam)^*)
=\BRST_{\bullet}(M_0(\lam)^*).
\end{align}
Let $\{F_p \BRST_{\bullet}(\gVerma{\lam}^*)\}$
be the corresponding increasing filtration of
$\BRST_{\bullet}(\gVerma{\lam}^*)$.

Because the filtration is compatible with
the action of 
the Hamiltonian
$-D$,
each $E^r_{p,q}$
decomposes 
into eigenspaces of $-D$
as complexes:
\begin{align}
 E^r_{p,q}=\bigoplus_{d\in -\bra \lam,D\ket +\Z_{\geq 0}}
(E^r_{p,q})_d.
\end{align}
It follows that
\begin{align}
 E^{\infty}_{p,q}=
\bigoplus_{d\in -\bra \lam, D\ket+\Z_{\geq 0}}
(E^{\infty}_{p,q})_d,
\end{align}
and each $(E^r)_d$ converges to $(E^{\infty})_d$.
In particular 
one has
\begin{align}
 \bigoplus_{p+q=n}(E_{p,q}^{\infty})_{\tp}
=\begin{cases}
\gr_F \BRST_0(\gVerma{\lam}^*)_{\tp}
%  H_0^{\Lie}(\sVerma{\slam))
&\text{if }p+q=0,\\
0&\text{if }p+q\ne 0.
 \end{cases}
\label{eq:E-infty-top}
\end{align}
by Proposition  \ref{Pro:BRST-top}.

Also
by \eqref{eq:2008-01-16-15-36}
this filtration is compatible with
the $\sigma_R$-twisted action 
of $\Wg$.
Hence
 each $E^r_{p,q}$ is a 
 Ramond twisted representation of $\Wg$,
and the differential $d^r$ is a 
morphism in $\Wg\adMod_{\sigma_R}$.
Therefore $\{F_p \BRST_{\bullet}(\gVerma{\lam}^*)\}$
is a filtration of Ramond twisted 
representations of $\Wg$,  and
the corresponding graded space
\begin{align*}
\gr^F \BRST_0(\gVerma{\lam}^*)
=\bigoplus_{p+q=0}E_{p,q}^{\infty}
\end{align*}is also an object of $\Wg\adMod_{\sigma_R}$.

\subsection{Step 4}
Consider the subcomplex
\begin{align}
 (\cE_{0}^{p,q})_{\tp}=(\cE_0^{p,q})_{\bra \lam,D\ket}
=F^p C^{p+q}_+(\lam)_{\tp}/F^{p+1}C^{p+q}_+(\lam)_{\tp}\nno\\
\cong \bigoplus_{\bra \lam-\mu,x_0\ket=p}C_+^{p+q}(\lam)_{\tp}^{\mu}
\nno
\end{align}
of $\cE_0^{p,q}$.
By \eqref{eq:2008-01-22-13-55}
one has
\begin{align}
\left((\cE_0^{p,\bullet})_{\tp}, \affQ{+}\right)
\cong
\bigoplus_{\bra \lam-\mu,x_0\ket=p}\left(
C_+^{p+q}(\lam)_{\tp}^{\mu},
( Q_+^{\st})_{(0)}
\right)
\end{align}
as complexes.

By definition $\cE_0^{p,\bullet}$
is spanned by the vectors
\begin{align}
 \wJ_{a_1}(m_1)\dots \wJ_{a_r}(m_r)\psi_{\beta_1}(n_1)
\dots \psi_{\beta_s}(n_s)c
\end{align}
with 
$c\in (\cE_0^{p,\bullet})_{\tp}$,
and $m_i, n_i<0$.
It follows that
each $D$-eigenspace
$(\cE_0^{p,\bullet})_d$ is finite-dimensional.
Thus
by Lemma \ref{Lem:easy-lemma},
\begin{align}
 E^0_{p,q}
(=(\cE_0^{p-1,q+1})^*)=
\sfD(\cE_0^{p-1,q+1}).
\label{eq:2008-01-15-15-55}
\end{align}

The following assertion follows immediately from 
\eqref{eq:2008-01-15-15-55}.
\begin{Pro}\label{Pro:2008-01-12}
 One has
$E^1_{p,q}=\sfD(\cE_1^{p-1,q+1})$,
or equivalently,
$\cE_1^{p,q}=\sfD(E^1_{p+1,q-1})$.
\end{Pro} 
The following assertion 
follows from Proposition
 \ref{Pro:2008-01-12}
by the inductive argument.
\begin{Pro}
\label{Pro:cE-to-E}
The spectral sequence $\cE^r$ 
converges to $\sfD(E_{\infty})$.
\end{Pro}

The proof of the following assertion is the same as that of Theorem \ref{Th:KW}.
\begin{Pro}\label{Pro:2008-01-11}
One has $\cE_1^{p,q}=0$ for  $p+q\ne 0$
and there is a linear  isomorphism 
\begin{align*}
  U(\sg^f[t\inv]t\inv) \* (\cE_1^{p,-p})_{\tp}\isomap  
\cE_1^{p,-p}
\end{align*}
of the form
\begin{align}
u_{i_1}(-n_1)\dots u_{i_r}(-n_r)\*v
\mapsto
 \vecW{i_1}_{-n_1}\dots \vecW{i_r}_{-n_r}v
\label{eq:2008-01-22}
\end{align}
with  a fixed PBW basis 
$\{u_{i_1}(-n_1)\dots u_{i_r}(-n_r)\}$ of $U(\sg^f\* \C[t\inv]t\inv)$.
\end{Pro}
Thanks to Proposition~\ref{Pro:2008-01-11}
the following assertion follows by induction.
\begin{Pro}\label{Pro:E_1}
 There exist isomorphisms of chain complexes
\begin{align*}
( \cE_r^ {p,q},d_r)\cong (U(\sg^f[t\inv]t\inv)\* (\cE_r^{p,q})_{\tp},1\* d^r)
\end{align*}
of the form
\eqref{eq:2008-01-22}
with $v\in (\cE_{r}^{p,q})_{\tp }$
for all $r\geq 1$.
Therefore one has the linear isomorphism
\begin{align*}
 \cE_{\infty}^{p.q}\cong U(\sg^f[t\inv]t\inv)
\* (\cE_{\infty}^{p,q})_{\tp }
\end{align*}
of the form
\eqref{eq:2008-01-22}
with $v\in (\cE_{\infty}^{p,q})_{\tp }$.
\end{Pro}

By (\ref{eq:E-infty-top})
and Proposition \ref{Pro:2008-01-12}
one has
\begin{align*}
 (\cE_{\infty}^{p,q})_{\tp }=\sfD((E^{\infty}_{p+1,q-1})_{\tp}
)=0
\quad \text{if }p+q\ne 0.
\end{align*}
By Proposition \ref{Pro:E_1}
this gives
$\cE_{\infty}^{p,q}=0$ if $p+q\ne 0$,
or equivalently,
\begin{align}
 E^{\infty}_{p,q}=0\quad \text{if }p+q\ne0.
\end{align}
This gives that
$\BRST_{n}(\gVerma{\lam}^*)=0$ for all $n\ne 0$.

Also,
from Proposition \ref{Pro:E_1}
if follows that
each
$\cE_{\infty}^{p,-p}$ is almost highest weight.
Therefore $E^{\infty}_{p,-p}
=\gr_p \BRST_0(\gVerma{\lam}^*)
=\sfD(E_{\infty}^{p-1,-p+1})$
is almost co-highest weight 
with $(E^{\infty}_{p,-p})_{\tp}
=(E^{\infty}_{p,-p})_{-\bra \lam,D\ket}$
(see Remark \ref{Rem:action-is-not-the-same}).
Hence
$\BRST_0(\gVerma{\lam}^*)$ is also co-highest weight.

This completes the proof of (ii) of Theorem \ref{Th:Vanishing-verma-dual}.
\qed

%%%%%%%%%%%%%%%%%%%%%%%%%%%%%%%%%
% References
%%%%%%%%%%%%%%%%%%%%%%%%%%%%%%%%%
\bibliographystyle{alpha}
\bibliography{math}

%\begin{thebibliography}{99}

%%%%%%%%%%%%%%%%%%%%%%%%%%%%%%%%%%%%%%%
%\bibitem{B1}
%T. Bridgeland,
%Equivalences of triangulated categories and Fourier-Mukai transforms,
%Bull. London Math. Soc., {\bf 31} (1999), 25--34.
%%%%%%%%%%%%%%%%%%%%%%%%%%%%%%%%%%%%%%%%
%\bibitem{Sa}
%I. Satake, 
%Algebraic Structures of Symmetric Domains,
%Publ. Math. Soc. Japan, {\bf 14}, 
%Iwanami Shoten, Tokyo; 
%Princeton University Press, Princeton, N.J., 1980.
%%%%%%%%%%%%%%%%%%%%%%%%%%%%%%%%%%%%%%%%%
%\bibitem{Vo} D. A. Vogan,  Jr.,
%A Langlands classification for unitary representations,
%In: Analysis on Homogeneous Spaces and Representation Theory of Lie Groups,
%Okayama-Kyoto,  1997,
%(eds. T. Kobayashi, M. Kashiwara, T. Matsuki, K. Nishiyama and T. Oshima),
%Adv. Stud. Pure Math., {\bf 26}, Math. Soc. Japan, 2000, pp. 299--324.
%%%%%%%%%%%%%%%%%%%%%%%%%%%%%%%%%%%%%%%%%%

%\end{thebibliography}

\end{document}